\documentclass{article}
\usepackage{amsmath,amsfonts,amssymb,amsthm}
\usepackage{fullpage}
\usepackage{hyperref}
\usepackage{pdfsync}
\usepackage{microtype}
\usepackage{enumitem}
\usepackage{graphics,graphicx,color}
\usepackage{multirow}
\usepackage{mathrsfs}
\usepackage{subfig}
\usepackage{float}
\usepackage{hyperref}
\usepackage{tabulary}
\usepackage{appendix}
\usepackage{thm-restate}
\usepackage{cleveref}

\newcommand{\Lap}{\text{Lap}}
\newcommand{\R}{\mathbb{R}}
\newcommand{\norm}[1]{\left\lVert#1\right\rVert}

\newcommand{\E}[1]{\mathbb{E}\left[#1\right]}
\newcommand{\abs}[1]{\lvert{#1}\rvert}
\newcommand{\set}[1]{\{{#1}\}}
\newcommand{\argmax}{\mathrm{argmax}}
\newcommand{\eps}{\epsilon}

\usepackage{xcolor}
\definecolor{DarkGreen}{rgb}{0.1,0.5,0.1}
\newcommand{\todo}[1]{\textcolor{DarkGreen}{[To do: #1]}}

\newcommand{\sk}[1]{\textcolor{blue}{[Sara: #1]}}
\newcommand{\wz}[1]{\textcolor{purple}{[Wanrong: #1]}}

\newcommand{\Offline}{{\textsc{OfflinePCPD}}}
\newcommand{\Online}{{\textsc{OnlinePCPD}}}
\newcommand{\AboveThresh}{{\textsc{AboveThresh}}}
\newcommand{\ReportMax}{{\textsc{ReportMax}}}

\usepackage{algorithmicx,algpseudocode, algorithm}

\newtheorem{theorem}{Theorem}
\newtheorem{lemma}{Lemma}
\newtheorem{corollary}[lemma]{Corollary}
\newtheorem{definition}{Definition}

\newtheorem{example}{Example}

\title{Differentially Private Change-Point Detection}

\author{Rachel Cummings\footnotemark[1] \and Sara Krehbiel\footnotemark[2] \and Yajun Mei\footnotemark[1] \and Rui Tuo\footnotemark[3] \and Wanrong Zhang\footnotemark[1]}

\begin{document}

\renewcommand{\thefootnote}{\fnsymbol{footnote}}
\footnotetext[1]{School of Industrial and Systems Engineering, Georgia Institute of Technology. Email: \texttt{\{rachelc, ymei, wanrongz\}@gatech.edu}.  R.C. supported in part by a Mozilla Research Grant.  Y.M. and W.Z. supported in part by NSF grant CMMI-1362876.}
\footnotetext[2]{Department of Math and Computer Science, University of Richmond. Email: \texttt{krehbiel@richmond.edu}. Supported in part by a Mozilla Research Grant.}
\footnotetext[3]{Academy of Mathematics and Systems Science, Chinese Academy of Sciences. Supported in part by NSF grant DMS-156443.}

\renewcommand{\thefootnote}{\arabic{footnote}}

\maketitle

\abstract{The {\em change-point detection problem} seeks to identify distributional changes at an unknown change-point~$k^*$ in a stream of data. This problem appears in many important practical settings involving personal data, including biosurveillance,  fault detection, finance, signal detection, and security systems. The field of {\em differential privacy} offers data analysis tools that provide powerful worst-case privacy guarantees. We study the statistical problem of change-point detection through the lens of differential privacy. We give private algorithms for both online and offline change-point detection, analyze these algorithms theoretically, and provide empirical validation of our results.}


\section{Introduction}

The {\em change-point detection problem} seeks to identify distributional changes at an unknown change-point~$k^*$ in a stream of data. 
The estimated change-point should be consistent with the hypothesis that the data are initially drawn from pre-change distribution $P_0$ but from post-change distribution $P_1$ starting at the change-point. This problem appears in many important practical settings, including biosurveillance,  fault detection, finance, signal detection, and security systems. For example, the CDC may wish to detect a disease outbreak based on real-time data about hospital visits, or smart home IoT devices may want to detect changes changes in activity within the home. In both of these applications, the data contain sensitive personal information.

The field of {\em differential privacy} offers data analysis tools that provide powerful worst-case privacy guarantees. Informally, an algorithm that is $\eps$-differentially private ensures that any particular output of the algorithm is at most $e^\eps$ more likely when a single data entry is changed. In the past decade, the theoretical computer science community has developed a wide variety of differentially private algorithms for many statistical tasks. The private algorithms most relevant to this work are based on the simple output perturbation principle that to produce an $\eps$-differentially private estimate of some statistic on the database, we should add to the exact statistic noise proportional to $\Delta/\eps$, where $\Delta$ indicates the {\em sensitivity} of the statistic, or how much it can be influenced by a single data entry.

We study the statistical problem of change-point problem through the lens of differential privacy. We give private algorithms for both online and offline change-point detection, analyze these algorithms theoretically, and then provide empirical validation of these results. 

\subsection{Related work}

The change-point detection problem originally arose from industrial quality control, and has since been applied in a wide variety of other contexts including climatology \cite{Lund:2002}, econometrics \cite{Bai:Perron:2003}, and DNA analysis \cite{Zhang:Siegmund:2012}. The problem is studied both in the \emph{offline setting}, in which the algorithm has access to the full dataset $X=\set{x_1,\dots,x_n}$ up front, and in the \emph{online setting}, in which data points arrive one at a time $X=\set{x_1,\dots}$.  Change-point detection  is a canonical problem in statistics that has been studied for nearly a century; selected results include \cite{shewhart:1931,page:1954,shiryaev:1963,roberts:1966,lorden:1971,pollak:1985, pollak:1987, moustakides:1986,lai:1995, lai:2001,kulldorff:2001,mei:2006a, mei:2008a, mei:2010,chan:2017}.

Our approach is inspired by the commonly used Cumulative Sum (CUSUM) procedure \cite{page:1954}. 
 It follows the generalized log-likelihood ratio principle, calculating $$\ell(k)=\sum_{i=k}^n \log \frac{P_1(x_i)}{P_0(x_i)}$$ for each $k\in[n]$ and declaring that a change occurs if and only if $\ell(\hat k)\ge T$ for MLE $\hat k = \argmax_k \ \ell(k)$ and appropriate threshold $T>0$. 
The existing change-point literature works primarily in the asymptotic setting when 
$k^{*}_{n} / n \to r$ for some $r \in (0,1)$ as $n\to\infty$ 
(see, e.g., \cite{Hinkley:1970,Carlstein:1988}). In contrast, we consider finite databases and provide the first accuracy guarantees for the MLE from a finite sample ($n<\infty$).

In offering the first algorithms for {\em private} change-point detection, 
we primarily use two powerful tools from the differential privacy literature.  \ReportMax\ \cite{dwork2014algorithmic} calculates noisy approximations of a stream of queries on the database and reports which query produced the largest noisy value. We instantiate this with partial log-likelihood queries to produce a private approximation of the the change-point MLE in the offline setting. 
\AboveThresh~\cite{DNRRV09} calculates noisy approximations of a stream of queries on the database iteratively and aborts as soon as a noisy approximation exceeds a specified threshold. We extend our offline results to the harder online setting, in which a bound on $k^*$ is not known a priori, by using \AboveThresh\ to identify a window of fixed size $n$ in which a change is likely to have occurred so that we can call our offline algorithm at that point to estimate the true change-point.

\subsection{Our results}

We use existing tools from differential privacy to solve the change-point detection problem in both offline and online settings, neither of which have been studied in the private setting before.

\paragraph{Private offline change-point detection.} We develop an offline private change-point detection algorithm \Offline\ (Algorithm~\ref{algo.offline}) that is accurate under one of two assumptions about the distributions from which data are drawn. As is standard in the privacy literature, we give accuracy guarantees  that bound the additive error of our estimate of the true change-point with high probability. 
Our accuracy theorem statements (Theorems~\ref{thm02} and \ref{thm.assumption1acc}) also provide guarantees for the non-private estimator for comparison. Since traditional statistics typically focuses on the the asymptotic consistency and unbiasedness of the estimator, ours are the first finite-sample accuracy guarantees for the standard (non-private) MLE. 
As expected, MLE accuracy decreases with the sensitivity of the measured quantity but increases as the pre- and post-change distribution grow apart. Interestingly, it is constant with respect to the size of the database. In providing MLE bounds alongside accuracy guarantees for our private algorithms, we are able to quantify the cost of privacy as roughly $D_{KL}(P_0||P_1)/\eps$. 

 We are able to prove $\eps$-differential privacy under the first distributional assumption, which is that the measured quantity has bounded sensitivity $\Delta(\ell)$, by instantiating the general-purpose \ReportMax\ algorithm from the privacy literature with our log-likelihood queries (Theorem~\ref{thm.assumption1priv}). Importantly and in contrast to our accuracy results, the distributional assumption need only apply to the hypothesized distributions from which data are drawn; privacy holds for arbitrary input databases. We offer a limited privacy guarantee for our second distributional assumption, ensuring that if an individual data point is drawn from one of the two hypothesized distributions, redrawing that data from either of the distributions will not be detected, regardless of the composition of the rest of the database (Theorem~\ref{thm.relaxedpriv}).

\paragraph{Private online change-point detection.} In \Online\ (Algorithm~\ref{algo2}), we extend our online results to the offline setting by using the \AboveThresh\ framework to first identify a window in which the change is likely to have happened and then call the offline algorithm to identify a more precise approximation of when it occurred. Standard $\eps$-differential privacy under our first distributional assumption follows from composition of the underlying privacy mechanisms (Theorem~\ref{thm.onlinepriv}).\footnote{We note that we can relax our distributional assumption and get a weaker privacy guarantee as in the offline setting if desired.} Accuracy of our online mechanism relies on appropriate selection of the threshold that identifies a window in which a change-point has likely occurred, at which point the error guarantees are inherited from the offline algorithm
 (Theorem~\ref{thm.onlineacc}).

\paragraph{Empirical validation.} Finally, we run several Monte Carlo experiments to validate our theoretical results for both the online and offline settings. We consider data drawn from Bernoulli and Gaussian distributions, which satisfy our first and second distributional assumptions, respectively. Our offline experiments are summarized in \Cref{fig:offline}, which shows that change-point detection is easier when $P_0$ and $P_1$ are further apart and harder when the privacy requirement is stronger ($\eps$ is smaller). Additionally, these experiments enhance our theoretical results, finding that \Offline\ performs well even when we relax the assumptions required for our theoretical accuracy bounds by running our algorithm on imperfect hypotheses $P_0$ and $P_1$ that are closer together than the true distributions from which data are drawn. \Cref{fig:online1} shows that \Online\ also performs well, consistent with our theoretical guarantees.


\section{Preliminaries}

Our work considers the statistical problem of change-point detection through the lens of differential privacy.  Section \ref{s.cpbackground} defines the change-point detection problem, and Section \ref{s.dpbackground} describes the differentially private tools that will be brought to bear.

\subsection{Change-point background}\label{s.cpbackground}



Let $X=\set{x_1,\dots,x_n}$ be $n$ real-valued data points.  The \emph{change-point detection problem} is parametrized by two distributions, $P_0$ and $P_1$.  The data points in $X$ are hypothesized to initially be sampled i.i.d. from $P_0$, but at some unknown change time $k^\ast \in [n]$, an event may occur (e.g., epidemic disease outbreak) and change the underlying distribution to $P_1$.  The goal of a data analyst is to announce that a change has occurred as quickly as possible after $k^\ast$.  Since the $x_i$ may be sensitive information---such as individuals' medical information or behaviors inside their home---the analyst will wish to announce the change-point time in a privacy-preserving manner.


In the standard non-private offline change-point literature, the analyst wants to test the null hypothesis $H_0:k^*=\infty$, where $x_1,\dots,x_n\sim_{\text{iid}}P_0$, against the composite alternate hypothesis $H_1:k^*\in[n]$, where $x_1,\dots,x_{k^*-1}\sim_{\text{iid}} P_0$ and $x_{k^*},\dots,x_n\sim_{\text{iid}} P_1$. The log-likelihood ratio of $k^*=\infty$ against $k^*=k$ is given by 
\begin{equation}\label{eqn01}\ell(k,X)=\sum_{i=k}^n \log \frac{P_1(x_i)}{P_0(x_i)}.\end{equation}
The maximum likelihood estimator (MLE) of the change time $k^*$ is given by
\begin{equation}\label{eqn03}\hat k(X) = \argmax_{k\in[n]} \ell(k,X).\end{equation}
When $X$ is clear from context, we will simply write $\ell(k)$ and $\hat{k}$.

An important quantity in our accuracy analysis will be the Kullback-Leibler distance between probability distributions $P_0$ and $P_1$, defined as $D_{KL}(P_1||P_0)= \int_{-\infty}^{\infty} P_1(x) \log\frac{P_1(x)}{P_0(x)} dx=\mathbb{E}_{x\sim P_1}[\log \frac{P_1(x)}{P_0(x)}]$. We always use $\log$ to refer to the natural logarithm,  and when necessary, we interpret $\log \frac{0}{0} = 0$.

We will measure the additive error  of our estimations of the true change point as follows. 

\begin{definition}[$(\alpha, \beta)$-accuracy]
A change-point detection algorithm that produces a change-point estimator $\tilde{k}(X)$ where a distribution change occurred at time $k^*$ is \emph{$(\alpha,\beta)$-accurate} if $\Pr[|\tilde{k}-k^\ast|<\alpha] \geq 1- \beta$, where the probability is taken over randomness of the algorithm and sampling of $X$.
\end{definition}


%


\subsection{Differential privacy background}\label{s.dpbackground}

Differential privacy bounds the maximum amount that a single data entry can affect analysis performed on the database.  Two databases $X,X'$ are \emph{neighboring} if they differ in at most one entry.

\begin{definition}[Differential Privacy \cite{DMNS06}]\label{def.dp}
An algorithm $\mathcal{M}: \mathbb{R}^n \rightarrow \mathcal{R}$ is \emph{$(\epsilon,\delta)$-differentially private} if for every pair of neighboring databases $X,X' \in \mathbb{R}^n$, and for every subset of possible outputs $\mathcal{S} \subseteq \mathcal{R}$,
\[ \Pr[\mathcal{M}(X) \in \mathcal{S}] \leq \exp(\epsilon)\Pr[\mathcal{M}(X') \in \mathcal{S}] + \delta. \]
If $\delta = 0$, we say that $\mathcal{M}$ is {\em $\epsilon$-differentially private}.
\end{definition}

One common technique for achieving differential privacy is by adding Laplace noise.  The \emph{Laplace distribution} with scale $b$ is the distribution with probability density function: $\Lap(x|b) = \frac{1}{2b} \exp\left(-\frac{|x|}{b}\right)$.  We will write $\Lap(b)$ to denote the Laplace distribution with scale $b$, or (with a slight abuse of notation) to denote a random variable sampled from $\Lap(b)$.

The \emph{sensitivity} of a function or query $f$ 
 is defined as 
  $\Delta (f) = \max_{\text{neighbors } X, X'} | f(X) - f(X') |$.  The Laplace Mechanism of \cite{DMNS06} takes in a function $f$, database $X$, and privacy parameter $\epsilon$, and outputs $f(X) + \Lap(\Delta (f)/\epsilon)$. 
  
  Our algorithms rely on two existing differentially private algorithms, \ReportMax~\cite{dwork2014algorithmic} and \AboveThresh~\cite{DNRRV09}. 
The \ReportMax\ algorithm takes in a collection of queries, computes a noisy answer to each query, and returns the index of the query with the largest noisy value.  We use this as the framework for  our offline private change-point detector \Offline\ in Section \ref{s.offline} to privately select the time $k$ with the highest log-likelihood ratio $\ell(k)$.

{\centering
\begin{minipage}{\linewidth}
\begin{algorithm}[H]
\caption{Report Noisy Max : \ReportMax($X, \Delta, \{f_1, \ldots, f_m\}, \epsilon$)}
\begin{algorithmic}
\State \textbf{Input:} database $X$, set of queries $\{f_1, \ldots, f_m\}$ each with sensitivity $\Delta$, privacy parameter $\epsilon$
	\For {$i=1,\ldots,m$}
	\State Compute $f_i(X)$
	\State Sample $Z_i \sim \Lap(\frac{\Delta}{\epsilon})$
	\EndFor
	\State Output $i^* =\underset{i \in [m]}{\mathrm{argmax}} \left(f_i(X) + Z_{i} \right)$
\end{algorithmic}
\end{algorithm}
\end{minipage}
}

\begin{theorem}[\cite{dwork2014algorithmic}]
\ReportMax\ is $(\epsilon, 0)$-differentially private. 
\end{theorem}

The \AboveThresh\ algorithm, first introduced by \cite{DNRRV09} and refined to its current form by \cite{dwork2014algorithmic}, takes in a potentially unbounded stream of queries, compares the answer of each query to a fixed noisy threshold, and halts when it finds a noisy answer that exceeds the noisy threshold.  We use this algorithm as a framework for our online private change-point detector \Online\ in Section \ref{s.online} when new data points arrive online in a streaming fashion.

{\centering
\begin{minipage}{\linewidth}
\begin{algorithm}[H]
\caption{Above Noisy Threshold: \AboveThresh($X, \Delta, \{f_1, f_2, \ldots \}, T, \epsilon$) }
\begin{algorithmic}
\State \textbf{Input:} database $X$, stream of queries $\{f_1, f_2, \ldots \}$ each with sensitivity $\Delta$, threshold $T$, privacy parameter $\epsilon$
\State Let $\hat{T} = T + \Lap(\frac{2\Delta}{\epsilon})$
\For {each query $i$}
  \State Let $Z_i \sim \Lap(\frac{4\Delta}{\epsilon})$
  \If {$f_i(X)+Z_i>\hat{T}$}
    \State Output $a_i=\top$
    \State Halt
  \Else
    \State Output $a_i=\bot$
  \EndIf
\EndFor
\end{algorithmic}
\end{algorithm}
\end{minipage}

\begin{theorem}[\cite{DNRRV09}]
\AboveThresh\ is $(\epsilon, 0)$-differentially private. 
\end{theorem}
}

\begin{theorem}[\cite{DNRRV09}]\label{thm.atacc}
For any sequence of $m$ queries $f_1,\ldots, f_m$ with sensitivity $\Delta$ such that $|\{i<m:f_i(X)\ge T-\alpha\}|=0$, \AboveThresh\ outputs with probability at least $1-\beta$ a stream of $a_1, \ldots, a_m \in \{\top, \bot\}$ such that $a_i=\bot$ for every $i\in[m]$ with $f(i)<T-\alpha$ and $a_i=\top$ for every $i\in[m]$ with $f(i)>T+\alpha$ as long as
\begin{align*}\alpha\ge \frac{8\Delta\log (2m/\beta)}{\epsilon}. \end{align*}
\end{theorem}

\subsection{Concentration inequalities} 

Our proofs will use the following bounds. 

	\begin{lemma}[Ottaviani's inequality \cite{van1996weak}] \label{lem.ott}
 For independent random variables $U_1,\ldots,U_m$, for $S_k=\sum_{i\in[k]}U_i$ for $k\in[m]$, and for $\lambda_1,\lambda_2>0$, we have
		$$\Pr\left[\max_{1 \le k\le m}|S_k|>\lambda_1+\lambda_2\right]\le \frac{\Pr \left[|S_m|>\lambda_1\right]}{1-\max_{1 \le k\le m} \Pr \left[|S_m-S_k|>\lambda_2\right]}.$$
	\end{lemma}

If we additionally assume the $U_j$ above are i.i.d. with mean 0 and take values from an interval of bounded length $L$, we can apply Hoeffding's inequality for the following corollary:

\begin{corollary}\label{cor.ott} For independent and identically distributed random variables $U_1,\dots,U_m$ with mean zero strictly bounded by an interval of length $L$ and for $S_k=\sum_{i\in[k]}U_i$ for $k\in[m]$, and for $\lambda_1,\lambda_2>0$, we have
$$
\Pr[\max_{k\in[m]}|S_k|>\lambda_1+\lambda_2]\le \frac{2\exp(-2\lambda_1^2/(mL^2))}{1-2\exp(-2\lambda_2^2/(mL^2))}.$$ 
\end{corollary}

\begin{lemma}[Bernstein inequality \cite{van1996weak}] \label{lem.bern}
	Let $Y_1,\ldots,Y_n$ be independent random variables with mean zero such that $\mathbb{E}\left[ e^{|Y_i|/M}-1-\frac{|Y_i|}{M}\right] M^2 \le \frac{1}{2}v_i$ for constants $M$ and $v_i$ and for  $i\in [n]$. Then
	$$\Pr[|Y_1+\ldots+Y_n|>x]\le 2\exp\left(-\frac{1}{2}\frac{x^2}{v+Mx}\right),$$
	for $v\ge v_1+\ldots+v_n.$
\end{lemma}

\begin{corollary}\label{cor.bern}
	For independent and identically distributed random variables $Y_1,\ldots,Y_n$ with  mean zero such that $\mathbb{E}\left[ e^{|Y_i|}-1-|Y_i|\right] \le \frac{1}{2}v$, for constant $v$ and  $i\in[n]$, and for $S_k=\sum_{i\in[k]}Y_i$ for $k\in[m]$, and for $\lambda_1,\lambda_2>0$, we have
	$$\Pr[\max_{k\in[m]}|S_k|>\lambda_1+\lambda_2]\le \frac{2\exp(-\lambda_1^2/(2mv+2\lambda_1))}{1-2\exp(-\lambda_2^2/(2mv+2\lambda_2))}.$$ 
\end{corollary}


\section{Offline private change-point detection}\label{s.offline}

In this section, we investigate the differentially private change point detection problem in the setting that $n$ data points $X=\set{x_1,\dots,x_n}$ are known to the algorithm in advance. Given two hypothesized distributions $P_0$ and $P_1$, our algorithm \Offline\ privately approximates the MLE $\hat k$ of the change time $k^*$. We provide accuracy bounds for both the MLE and the output of our algorithm under two different assumptions about the distributions from which the data are drawn, summarized in \Cref{tab.offlineresults}.  

\begin{center}
\begin{minipage}{.9\linewidth}
\begin{table}[H]
\begin{center}\begin{tabular}{ccc} \hline
 Assumption & \phantom{\qquad\qquad}MLE\phantom{\qquad\qquad} & \Offline  \\ \hline 
$A:=\Delta(\ell)<\infty$ & $\frac{2A^2}{C^2}\log \frac{32}{3\beta}$ & $\max\left\{\frac{8A^2}{C^2}\log \frac{64}{3\beta},\, \frac{4A}{C\epsilon} \log \frac{16}{\beta}\right\}$ \vspace{.5em}\\
$A:=A_\delta<\infty$ & $\frac{67}{C_M^2}\log \frac{64}{3\beta}$ & $\max\left\{\frac{262}{C_M^2}\log \frac{128}{3\beta}, \frac{2A\log (16/\beta)}{C_M\eps}\right\}$ \vspace{.5em}\\ \hline
\end{tabular}
\caption{ \small Summary of non-private and private offline accuracy guarantees under $H_1$. \qquad\qquad The expressions $\Delta(\ell)$, $A_\delta$, $C$, and $C_M$ are defined in (\ref{eqnconsA}), (\ref{eqnconsA2}), (\ref{eq.C}), (\ref{eqnthm02eq3b}), resp.
}\label{tab.offlineresults}
\end{center}
\end{table}
\end{minipage}
\end{center}

The first assumption essentially requires that $P_1(x)/P_0(x)$ cannot be arbitrarily large or arbitrarily small for any $x$. We note that this assumption is not satisfied by several important families of distributions, including Gaussians. The second assumption, motivated by the $\delta>0$ relaxation of differential privacy, instead requires that the $x$ for which this log ratio exceeds some bound $A_\delta$ have probability mass at most~$\delta$.

Although the accuracy of  \Offline\ only holds under the change-point model's alternate hypothesis $H_1$, it  is $\eps$-differentially private for any {\em hypothesized} distributions $P_0,P_1$ with finite $\Delta(\ell)$ and privacy parameters $\eps>0,\delta=0$ {\em regardless of the distributions from which $X$ is drawn}. We offer a similar but somewhat weaker privacy guarantee when $\Delta(\ell)$ is infinite but $A_\delta$ is finite, which roughly states that a data point sampled from either $P_0$ or $P_1$ can be replaced with a fresh sample from either $P_0$ or $P_1$ without detection. 



\subsection{Offline algorithm}\label{sec:off1}

Our proposed offline algorithm \Offline\ applies the report noisy max algorithm  \cite{dwork2014algorithmic} to the change-point problem by adding noise to partial log-likelihood ratios $\ell(k)$ used to estimate the change point MLE $\hat k$. The algorithm chooses Laplace noise parameter $A/\eps$ depending on input hypothesized distributions $P_0,P_1$ and privacy parameters $\eps,\delta$ and then outputs
\begin{eqnarray} \label{eqn04}
\tilde{k} =\underset{1\le k \le n}{\mathrm{argmax}}\set{\ell(k)+ Z_{k}} .
\end{eqnarray}
Our algorithm can be easily modified to additionally output an approximation of $\ell(\tilde k)$ and incur $2\eps$ privacy cost by composition.

{\centering
\begin{minipage}{\linewidth}
\begin{algorithm}[H]
	\caption{Offline private change-point detector : \Offline($X, P_0, P_1,\epsilon, \delta,n$)}
	\begin{algorithmic}
	\State \textbf{Input:} database $X$, distributions $P_0,P_1$, privacy parameters $\eps,\delta$, database size $n$
		\If {$\delta=0$}
		\State Set $A= \max_{x}\ \log  \frac{P_1(x)}{P_0(x)} - \min_{x'}\ \log  \frac{P_1(x')}{P_0(x')}$ \Comment set $A=\Delta \ell$ as in \eqref{eqnconsA}
		\Else
		\State Set $A= \min \{ t \; : \; \max_{i=0,1}\Pr_{x\sim  P_i}[2|\log \frac{P_1(x)}{P_0(x)}|> t] < \delta/2 \}$ \Comment set $A= A_\delta$ as in \eqref{eqnconsA2}
		\EndIf
		\For {$k=1,\ldots,n$}
		\State Compute $\ell(k)=\sum_{i=k}^n \log \frac{P_1(x_i)}{P_0(x_i)}$
		\State Sample $Z_k \sim \Lap(\frac{A}{\epsilon})$
		\EndFor
		\State Output $\tilde{k} =\underset{1\le k \le n}{\mathrm{argmax}} \set{\ell(k)+ Z_{k} }$ \Comment Report noisy argmax
	\end{algorithmic}\label{algo.offline}
\end{algorithm}
\end{minipage}
}



In the change-point  or statistical process control (SPC) literature, when the pre- and post- change distributions are unknown in practical settings, researchers often choose hypotheses $P_0,P_1$ with the smallest justifiable 
 distance. 
While it is easier to detect and accurately estimate a larger change, larger changes are often associated with a higher-sensitivity MLE, requiring more noise (and therefore additional error) to preserve privacy. We propose that practitioners using our private change point detection algorithm choose input hypotheses accordingly. This practical setting is considered in our numerical studies, presented in Section~\ref{s.sim}. 

In the case that $\delta=0$, we sample Laplace noise directly proportional to the sensitivity of the partial log-likelihood ratios we compute: 
\begin{eqnarray}
\Delta \ell &= & \max_{\substack{k\in[n],X,X'\in\R^n \\ \norm{X - X'}_1=1}} \norm{\ell(k,X)- \ell(k,{X'})}_1 = \max_{x\in\R}\ \log  \frac{P_1(x)}{P_0(x)} - \min_{x'\in\R}\ \log  \frac{P_1(x')}{P_0(x')}.  \label{eqnconsA}
\end{eqnarray}


The algorithm should not be invoked with $\delta=0$ unless $\Delta(\ell)$ is finite. In the case that $\ell$ has infinite sensitivity, we instead allow the user to select a privacy parameter $\delta>0$ and identify a value $A_\delta$ for which most values of $x\sim P_0,P_1$ have bounded log-likelihood ratio:
\begin{eqnarray} \label{eqnconsA2}
A_\delta = \min \left\{ t \; : \; \max_{i=0,1}\Pr_{x\sim  P_i}\left[2\abs{\log \frac{P_1(x)}{P_0(x)}}> t\right] < \delta/2 \right\}. 
\end{eqnarray} 

As a concrete canonical example, $\Delta(\ell)$ is unbounded for two Gaussian distributions, but $A_\delta$ is bounded for Gaussians with different means as follows: 

\begin{example}\label{ex.gauss}
 For $P_0=\mathcal{N}(0,1)$, $P_1=\mathcal{N}(\mu,1)$, and $\delta>0$, we have 
 $A_\delta=2\mu[\Phi^{-1}(1-\delta/2)+\mu/2],$
 where $\Phi$ is the cumulative distribution function (CDF) of the standard normal distribution.
\end{example}



\subsection{Theoretical properties under the uniform bound assumption} \label{sec:off2}
In this subsection, we prove privacy and accuracy of \Offline\ when $\delta=0$ and $P_0,P_1$ are such that $\Delta(\ell)$ is finite. Note that if $\Delta(\ell)$ is infinite, then the algorithm will simply add noise with infinite scale and will still be differentially private.


\begin{theorem}\label{thm.assumption1priv} For arbitrary data $X$, \Offline$(X,P_0,P_1,\eps,0)$ is $(\eps,0)$-differentially private. 
\end{theorem}

\begin{proof} Privacy follows by instantiation of \ReportMax\ \cite{dwork2014algorithmic} with queries $\ell(k)$ for $k\in[n]$, which have sensitivity $A=\Delta(\ell)$; this proof is included for completeness.

		
	Fix any two neighboring databases $X,X'$ that differ on index $j$. For any $k\in[n]$, denote the respective partial log-likelihood ratios as $\ell(k)$ and $\ell'(k)$. By (\ref{eqn01}), we have
	\begin{equation} \label{eqnThm1prof2}
	\ell^\prime(k)=\ell(k)+\Delta \mathbb{I}\{j \ge k\} \qquad \mbox{ with } \quad \Delta =\log\frac{P_1(x_j^\prime)}{P_0(x_j^\prime)}-\log\frac{P_1(x_j)}{P_0(x_j)}.
	\end{equation}
	
	Next, for a given $1 \le i \le n$, fix $Z_{-i}$, a draw from $[\Lap(A/\epsilon)]^{n-1}$ used for all the noisy log likelihood ratio values except the $i$th one. We will bound from above and below the ratio of the probabilities that the algorithm outputs $\tilde k = i$ on inputs $X$ and $X'$. 
	Define the minimum noisy value in order for $i$ to be select with $X$:
	\begin{align*}
	Z^\ast_i &=\min \set{{Z_i}: \ell(i)+Z_i>\ell(k)+Z_k \quad \forall k\neq i} 
	\end{align*}
	
	
	If $\Delta <0$, then for all $k\ne i$ we have
	$$\ell^\prime(i)+A+Z^\ast_i \ge \ell(i)+Z^\ast_i>\ell(k)+Z_k \ge \ell^\prime(k)+Z_k.$$	
	
	If $\Delta\ge 0$, then for all $k\ne i$ we have
	$$\ell^\prime(i)+Z^\ast_i \ge \ell(i)+Z^\ast_i>\ell(k)+Z_k \ge \ell^\prime(k)-A+Z_k.$$
	Hence, $Z_i^\prime \ge Z^\ast_i+ A$ ensures that the algorithm outputs $i$ on input ${X'}$, and the theorem follows from the following inequalities for any fixed $Z_{-i}$, with probabilities over the choice of $Z_i\sim \Lap(A/\eps)$. 
	\begin{eqnarray*}
		\Pr[\tilde{k}=i\mid{X'}, Z_{-i}] \ge \Pr[ Z_i^\prime \ge Z^\ast_i+ A\mid Z_{-i}] \ge e^{-\epsilon}\Pr[ Z_i \ge Z^\ast_i\mid Z_{-i}] = e^{-\epsilon}\Pr[\tilde{k}=i \mid { X}, Z_{-i}]
	\end{eqnarray*}
	\end{proof}

Next we provide accuracy guarantees of the standard (non-private) MLE   $\hat k$ and the output $\tilde k$ of our private algorithm \Offline\ when the data are drawn from $P_0,P_1$ with true change point $k^{*} \in (1, n)$. By providing both bounds, Theorem \ref{thm02}  quantifies the cost of requiring privacy in change point detection.

Our result for the standard (non-private) MLE is 
the first finite-sample accuracy guarantee for this estimator. 
Such non-asymptotic properties have not been previously studied in traditional statistics, which typically focuses on consistency and unbiasedness of the estimator, with less attention to the convergence rate.  We show that the additive error of the MLE  is constant with respect to the sample size, which means that the convergence rate is $O_P(1)$. That is, it converges in probability to the true change-point $k^*$ in constant time.

Note that accuracy depends on two measures $A$ and $C$ of the distances between distributions $P_0$ and $P_1$. Accuracy both of MLE $\hat k$ and \Offline\ output $\tilde k$ is best for distributions for which $A=\Delta(\ell)$ is small relative to KL-divergence, which is consistent with the intuition that larger changes are easier to detect but output sensitivity degrades the robustness of the estimator and requires more noise for privacy, harming accuracy.

A technical challenge that arises in proving accuracy of the private estimator is that the $x_i$ are not identically distributed when the true change-point $k^{*} \in (1,n]$, and so the partial log-likelihood ratios $\ell(k)$ are dependent across $k$.  Hence we need to investigate the impact of adding i.i.d. noise draws to a sequence of $\ell(k)$ that may be neither independent nor identically distributed.  Fortunately, the differences $\ell(k) - \ell(k+1) = \log\frac{P_1(x_k)}{P_0(x_k)}$ are piecewise i.i.d. This property is key in our proof. 
Moreover, we show that we can divide the possible outputs of the algorithm into regions that of doubling size with exponentially decreasing probability of being selected by the algorithm, resulting in accuracy bounds that are independent of the number of data points $n$.


\begin{restatable}{theorem}{offaccA} \label{thm02}
	For hypotheses $P_0,P_1$ such that $\Delta(\ell)<\infty$ and $n$ data points $X$ drawn from $P_0,P_1$ with true change time $k^*\in(1,n]$, the MLE $\hat k$ is $(\alpha,\beta)$-accurate for any $\beta>0$ and 
	\begin{equation}  \label{eqnthm02eq2}
	\alpha=\frac{2A^2}{C^2}\log \frac{32}{3\beta}.
	\end{equation}
	
	For hypotheses and data drawn this way with privacy parameter $\eps>0$, \Offline$(X,P_0,P_1,\eps,0,n)$ is $(\alpha,\beta)$-accurate for any $\beta>0$ and 
	\begin{equation} \label{eqnthm02eq1}
	\alpha=\max\left\{\frac{8A^2}{C^2}\log \frac{64}{3\beta},\, \frac{4A}{C\epsilon} \log \frac{16}{\beta}\right\}.
	\end{equation}

	In both expressions, $A=\Delta(\ell)$ and $C=\min\set{D_{KL}(P_1||P_0), D_{KL}(P_0||P_1)}$.
\end{restatable}

\begin{proof} 

Our goal is to find some expression for $\alpha$ such that we can bound the probability of the bad event that \Offline\ outputs $\tilde k$ such that $\abs{\tilde k-k^*}> \alpha$ with probability at most $\beta$, where $k^*$ is the true change point. The first half of our analysis will yield another bound giving accuracy of the MLE $\hat k$.

Our proof is structured around the following observation. The algorithm only outputs a particular incorrect $\tilde k\ne k^*$ if there exists some $k$ in with $\ell(k)+Z_k > \ell(k^*)+Z_{k^*}$ for a set of random noise values $\set{Z_k}_{k\in[n]}$ selected by the algorithm. For the algorithm to output an incorrect value, there must either be a $k$ that nearly beats the true change point on the noiseless data or there must be a $k$ that receives much more noise than $k^*$. Intuitively, this captures the respective scenarios that unusual data causes non-private ERM to perform poorly and that unusual noise draws causes our private algorithm to perform poorly. 

Given some true change-point $k^*$ and error tolerance $\alpha>0$, we can partition the set of bad possible outputs $k$ into sub-intervals of exponentially increasing size as follows. For $i\ge 1$, let
\begin{align*}
R_i^- &= [k^*-2^i\alpha, k^*-2^{i-1}\alpha) \\
R_i^+ &=  (k^*+2^{i-1}\alpha,k^*+2^i\alpha] \\
R_i &= R_i^-\cup R_i^+
\end{align*}

Then for any range-specific thresholds $t_i$ for $i\ge 1$, our previous observations allow us to bound the probability of the bad event as follows: 
\begin{align}\label{eq.bad}
\Pr[\abs{\tilde k - k^*}>\alpha] &\le \sum_{i\ge 1}\Pr[\max_{k\in R_i} \set{\ell(k)-\ell(k^*)}>-t_i] + \sum_{i\ge 1}\Pr[\max_{k\in R_i}\set{Z_k-Z_{k^*}}\ge t_i]
\end{align}

We bound each term in the above expression separately for $t_i=2^{i-2}\alpha C$. For accuracy of the non-private MLE, we will set $\alpha$ to ensure that the first term is at most $\beta$. For accuracy of the private algorithm, we will set $\alpha$ to ensure that each term is at most $\beta/2$. 
The first and more difficult task requires us to reason about the probability that the log-likelihood ratios for the data are not too far away from their expectation. Although the $\ell(k)$ are not independent, their pairwise differences $\ell(k+1)-\ell(k)$ are, so we can apply our corollary of Ottaviani's inequality to bound the probability that $\ell(k)$ significantly exceeds $\ell(k^*)$ by appropriately defining several random variables corresponding to a data stream $X$ drawn according to the change-point model. 

Specifically, we can decompose the empirical log-likelihood difference between the true change-point $k^*$ and any candidate $k$ into the sum of i.i.d. random variables with mean zero and the expected value of this difference as follows:

\begin{align*}
U_j &= \begin{cases}
- \log \frac{P_0(x_j)}{P_1(x_j)}+D_{KL}(P_0||P_1), & j<k^* \\
- \log \frac{P_1(x_j)}{P_0(x_j)}+D_{KL}(P_1||P_0), & j\ge k^* 
\end{cases}
\\
\ell(k)-\ell(k^*) &= \begin{cases}
\sum_{j=k}^{k^*-1}U_j - (k^*-k)D_{KL}(P_0||P_1), & k< k^* \\
\sum_{j=k^*}^{k-1}U_j - (k-k^*)D_{KL}(P_1||P_0), & k\ge k^* 
\end{cases}
\end{align*}

We also define random variable $S_m$ to denote the sum of $m$ i.i.d. random variables as follows, noting that $S_m$ is distributed like $\sum_{j=k^*+m}^{k^*-1} U_j$ for $m<0$ and like $\sum_{j=k^*}^{k^*+m-1} U_j$ for $m>0$.
\begin{align*}
S_m &= \begin{cases} \sum_{k^*+m\le j<k^*} U_j, &m<0 \\ 
\sum_{k^*\le j<k^*+m} U_j & m> 0 \end{cases}
\end{align*}

With these random variables, we bound each term in the first set of terms in~(\ref{eq.bad}) for any $i\ge 1$ and threshold $t_i=2^{i-2}\alpha C$ as follows:
\begin{align}
\Pr[\max_{k\in R_i} &\set{\ell(k)-\ell(k^*)}>-2^{i-2}\alpha C] \notag \\
&\le \Pr[\max_{k\in R_i^-}\set{\sum_{j=k}^{k^*-1}U_j-(k^*-k)D_{KL}(P_0||P_1)}>-2^{i-2}\alpha C] \notag \\
&\quad +\Pr[\max_{k\in R_i^+}\set{\sum_{j=k^*}^{k-1}U_j-(k-k^*)D_{KL}(P_1||P_0)}>-2^{i-2}\alpha C] \notag \\
&\le \Pr[\max_{k\in[2^{i-1}\alpha]} \abs{S_{-k}} > 2^{i-2}\alpha C]  + \Pr[\max_{k\in[2^{i-1}\alpha]} \abs{S_{k}} > 2^{i-2}\alpha C] \notag \\
&\le \frac{4\cdot \exp(-2^{i-4}\alpha C^2/A^2)}{1-2\cdot \exp(-2^{i-4}\alpha C^2/A^2)}  \label{eq.bound1} \\
&\le 8 \exp(-2^{i-4} \alpha C^2/A^2) \label{eq.bound1b} \\
&= 8 \left(\exp(\frac{-\alpha C^2}{8A^2})\right)^{2^{i-1}} \notag
\end{align}

where \eqref{eq.bound1}~follows from an application of Corollary~\ref{cor.ott} with $\lambda_1=\lambda_2=2^{i-3}\alpha C$ and $L=A$, and the denominator can be simplified as in \eqref{eq.bound1b} under the assumption that $\alpha \geq \frac{8 A^2 \log 4}{C^2}$ to simplify the denominator, which is satisfied by our final bounds. 

We now consider the sum of these terms over all $i$, which will be needed for the final bound on Equation \eqref{eq.bad}.  We note that this sum is bounded above by a geometric series with ratio $\exp(-\alpha C^2/(8 A^2))$ since $2^{i-1}\ge i$, yielding the second inequality. 
Then the same assumed lower bound on $\alpha$ is used to simplify the denominator as in \eqref{eq.bound1b}:


\begin{align}
\sum_{i\geq 1} \Pr[ \max_{k \in R_i} \{ \ell(k) - \ell(k^*) \} > -2^{i-2}\alpha C] & \leq 8 \sum_{i\geq 1} \left(\exp(\frac{-\alpha C^2}{8A^2})\right)^{2^{i-1}} \notag \\
& \leq 8 \sum_{i\geq 1} \left(\exp(\frac{-\alpha C^2}{8A^2})\right)^{i} \notag \\
& \leq \frac{8\exp(\frac{-\alpha C^2}{8A^2})}{1-\exp(\frac{-\alpha C^2}{8A^2})} \notag \\
& \leq \frac{32}{3} \exp\left(\frac{-\alpha C^2}{8A^2}\right) \label{eq.bound3}
\end{align}


The first term in \eqref{eqnthm02eq1}  in the theorem statement ensures that the expression above is bounded by $\beta/2$, as is required for the private algorithm. 

For non-private MLE, we bound each term in the first set of terms in~(\ref{eq.bad}) for any $i\ge 1$ and threshold $t_i=0$ as follows:
\begin{align}
\Pr[\max_{k\in R_i} &\set{\ell(k)-\ell(k^*)}>0] \notag \\
&\le \Pr[\max_{k\in R_i^-}\set{\sum_{j=k}^{k^*-1}U_j-(k^*-k)D_{KL}(P_0||P_1)}>0] \notag \\
&\quad +\Pr[\max_{k\in R_i^+}\set{\sum_{j=k^*}^{k-1}U_j-(k-k^*)D_{KL}(P_1||P_0)}>0] \notag \\
&\le \Pr[\max_{k\in[2^{i-1}\alpha]} \abs{S_{-k}} > 2^{i-1}\alpha C]  + \Pr[\max_{k\in[2^{i-1}\alpha]} \abs{S_{k}} > 2^{i-1}\alpha C] \notag \\
&\le \frac{4\cdot \exp(-2^{i-2}\alpha C^2/A^2)}{1-2\cdot \exp(-2^{i-2}\alpha C^2/A^2)}  \label{eq.bound1} \\
&\le 8 \exp(-2^{i-2} \alpha C^2/A^2) \label{eq.bound1b} \\
&= 8 \left(\exp(\frac{-\alpha C^2}{2A^2})\right)^{2^{i-1}} \notag
\end{align}
Summing these terms over all $i$, 
\begin{align}
\sum_{i\geq 1} \Pr[ \max_{k \in R_i} \{ \ell(k) - \ell(k^*) \} > 0] & \leq 8 \sum_{i\geq 1} \left(\exp(\frac{-\alpha C^2}{2A^2})\right)^{2^{i-1}} \notag \\
& \leq 8 \sum_{i\geq 1} \left(\exp(\frac{-\alpha C^2}{2A^2})\right)^{i} \notag \\
& \leq \frac{8\exp(\frac{-\alpha C^2}{2A^2})}{1-\exp(\frac{-\alpha C^2}{2A^2})} \notag \\
& \leq \frac{32}{3} \exp\left(\frac{-\alpha C^2}{2A^2}\right). \label{eq.bound3}
\end{align}
For $\alpha$ as in \eqref{eqnthm02eq2} in the theorem statement, the expression above is bounded by $\beta$, completing the accuracy proof for the non-private MLE.


Next we bound the second set of terms in~(\ref{eq.bad}), controlling the probability that large noise draws cause large inaccuracies for the private algorithm. Since each $Z_k$ and $Z_{k^*}$ are independent draws from a Laplace distribution with parameter $A/\eps$, this bound follows from a union bound over all indices in $R_i$ and the definition of the Laplace distribution:  
	\begin{align*}
	\Pr[\max_{k\in R_i} \set{Z_k-Z_{k^*}}\ge 2^{i-2}\alpha C] &\le \Pr[2\max_{k\in R_i} |Z_k|\ge 2^{i-2}\alpha C] \\
	&\le 2^{i}\alpha \Pr[ |\Lap(A/\eps)| \ge 2^{i-3} \alpha C]  \\
	& \le 2^{i}\alpha\cdot \exp(-2^{i-3}\alpha C \eps / A) \\
	& = 2^{i}\alpha \left( \exp(\frac{-\alpha C \eps}{4A})\right)^{2^{i-1}} 
	\end{align*}

Then by summing over all ranges and assuming in \eqref{eq.Zdenom} that $\alpha\ge \frac{4A\ln 2}{C\eps}$ to simplify the denominator, we obtain a bound on the probability of large noise applied to any possible $k$ far from $k^*$.


\begin{align}
\sum_{i\geq 1} \Pr[ \max_{k \in R_i} \{ Z_k - Z_{k^*} \} > 2^{i-2}\alpha C] & \le \alpha \sum_{i\ge 1} 2^{i} (\exp(-\alpha C\eps/(4A)))^{2^{i-1}} \notag \\
&\le \alpha 2\sum_{i\ge 1} i (\exp(-\alpha C\eps/(4A)))^i \notag \\
&= \alpha 2\frac{\exp(-\alpha C\eps/(4A))}{(1-\exp(-\alpha C\eps/(4A)))^2} \notag \\
&\le 8\alpha \exp(-\alpha C\eps/(4A)) \label{eq.Zdenom}
\end{align}
Since $x/2\ge \ln x$, requiring $\alpha\ge \frac{4A\log (16/\beta)}{C\eps}$ suffices to ensure that \eqref{eq.Zdenom} is at most $\beta/2$ as required. By Inequality~\ref{eq.bad}, this guarantees that $\Pr[\abs{\tilde k - k^*}>\alpha]\le \beta$ for the assumed ranges of $\alpha$ captured in Equation~\eqref{eqnthm02eq1} in the theorem statement, completing the proof.

\end{proof}



\subsection{Relaxing  uniform bound  assumptions}\label{sec:off3}

In this subsection, we prove accuracy and a limited notion of privacy for \Offline\ when $\delta>0$ and $P_0,P_1$ are such that $A_\delta$ is finite. Since we are no longer able to uniformly bound $\log P_1(x)/P_0(x)$, these accuracy results include worse constants than those in Section \ref{sec:off2}, but the relaxed assumption about $P_0,P_1$ makes the results applicable to a wider range of distributions, including Gaussian distributions (see Example \ref{ex.gauss}). Note of course that for some pairs of very different distributions, such as distributions with non-overlapping supports, the assumption that $A_\delta<\infty$ may still fail. A true change point $k^*$ can always be detected with perfect accuracy given $x_{k^*-1}$ and $x_{k^*}$, so we should not expect to be able to offer any meaningful privacy guarantees for such distributions.  

By similar rationale, relaxing the uniform bound assumption means that we may have a single data point $x_j$ that dramatically increases $\ell(k)$ for $k\ge j$, so we cannot add noise proportional to $\Delta(\ell)$ and privacy no longer follows from that of \ReportMax. Instead we offer a weaker notion of privacy in Theorem~\ref{thm.relaxedpriv} below. As with the usual definition of differential privacy, we guarantee that the output of our algorithm is similarly distributed on neighboring databases, only our notion of neighboring databases depends on the hypothesized distributions. Specifically, the a single entry in $X$ drawn from either $P_0$ or $P_1$ may be replaced without detection by another entry drawn from either $P_0$ or $P_1$, even if the rest of the database is arbitrary.

\begin{theorem}\label{thm.relaxedpriv} For any $\eps,\delta>0$, any hypotheses $P_0,P_1$ such that $A_\delta<\infty$, any index $j\in[n]$, any $i,i'\in\set{0,1}$, and any $x_1,\dots,x_{j-1},x_{j+2},\dots,x_n$, let $X_i=\set{x_1,\dots,x_n}$ denote the random variable with $x_j\sim P_i$ and let $X'_{i'}=\set{x_1,\dots, x_{j-1},x'_j,x_{j+1},\dots,x_n}$ denote the random variable with $x'_j\sim P_{i'}$. 
Then for any $S\subseteq [n]$, we have  
\begin{align*}
\Pr[\Offline(X_i,P_0,P_1,\eps,\delta,n)\in S] &\le \exp(\eps)\cdot \Pr[\Offline(X'_{i'},P_0,P_1,\eps,\delta,n)\in S] +\delta, 
  \end{align*}
where the probabilities are over the randomness of the algorithm and of $X_i,X'_{i'}$.
\end{theorem}

\begin{proof} 
	
Define the event that the log-likelihood ratios of $x_j,x_j'$ as in the theorem statement are bounded by $A_\delta$ as follows: 
	$$E_{\delta}:=\left\{\left|\log \frac{P_1(x_j)}{P_0(x_j)}-\log \frac{P_1(x'_j)}{P_0(x'_j)}\right|<A_\delta \right\}. $$

	Let $\tilde{k}=\Offline(X_i,P_0,P_1,\eps,\delta,n),\tilde{k}'=\Offline(X'_{i'},P_0,P_1,\eps,\delta,n)$. Then by Theorem~\ref{thm.assumption1priv} and the observation that $\Pr[E_\delta^c]<\delta$ by definition of $A_\delta$, we have that for any $S\subseteq[n]$,  
	\begin{eqnarray*}
	\Pr[\tilde{k}\in S]&\leq& \Pr[\tilde{k}\in S|E_\delta]\Pr[E_\delta]+\Pr[E_\delta^c]\\
	&\leq& \exp(\epsilon)\Pr[\tilde{k}'\in S|E_\delta]\Pr[E_\delta]+\delta\\
	&\leq &\exp(\epsilon)\Pr[\tilde{k}'\in S]+\delta.
	\end{eqnarray*}
%
\end{proof}

Allowing $\Delta(\ell)$ to be infinite precludes our use of Hoeffding's inequality as in Theorem~\ref{thm02}. The main idea in the proof, however, can be salvaged by decomposing the change into a change from $P_0$ to the average distribution $(P_0+P_1)/2$ and then the average distribution to $P_1$. Correspondingly, we will use $C_M$, an alternate distance measure between $P_0$ and $P_1$, defined below next to $C$ from the previous section for comparison:
\begin{align} 
C&=\min\left\{D_{KL}(P_0||P_1), D_{KL}(P_1||P_0)\right\} \label{eq.C} \\
C_M&=\min\left\{D_{KL}(P_0||\frac{P_0+P_1}{2}), D_{KL}(P_1|| \frac{P_0+P_1}{2})\right\} =\min_{i=0,1} \mathbb{E}_{x\sim P_i}\left[\log \frac{2P_i(x)}{P_0(x)+P_1(x)}\right] \label{eqnthm02eq3b}
\end{align}
Because $(2 P_i)/(P_0+P_1)\leq 2$, we have $0 \leq D_{KL}(P_i||(P_0+P_1)/2)\leq \log 2,$ and thus the constant $C_M$ in (\ref{eqnthm02eq3b}) is well-defined.



\begin{restatable}{theorem}{offaccB}\label{thm.assumption1acc}	
For $\delta>0$ and hypotheses $P_0,P_1$ such that $A_\delta<\infty$ and $n$ data points $X$ drawn from $P_0,P_1$ with true change time $k^*\in(1,n)$, the MLE $\hat k$ is $(\alpha,\beta)$-accurate for any $\beta >0$ and 
	\begin{equation}  \label{eqnthm02eq4}
	\alpha=\frac{67}{C_M^2}\log \frac{64}{3\beta}.
	\end{equation}

For hypotheses and data drawn this way with privacy parameter $\eps>0$, \Offline$(X,P_0,P_1,\eps,\delta,n)$ is $(\alpha,\beta)$-accurate for any $\beta>0$ and 
	\begin{equation} \label{eqnthm02eq3}
\alpha=\max\{\frac{262}{C_M^2}\log \frac{128}{3\beta}, \frac{2A\log (16/\beta)}{C_M\eps}\}.
	\end{equation}
	In both expressions, $A=A_\delta$ and $C_M=\min\left\{D_{KL}(P_0||\frac{P_0+P_1}{2}), D_{KL}(P_1|| \frac{P_0+P_1}{2})\right\}$.
\end{restatable}

\begin{proof}
		The general framework of this proof is similar to that of Theorem \ref{thm02}, but the main difference is that  Hoeffding's inequality is not applicable in this general setting, since we allow $\Delta(\ell)$ to be unbounded. The main idea in this proof is to consider the alternative log-likelihood ratio using the average distribution $(P_0+P_1)/2$, in which Bernstein inequality can be applied. 
		
		Following the notation from Theorem \ref{thm02}, given some true change-point $k^*$ and error tolerance $\alpha>0$, we can partition the set of bad possible outputs $k$ into sub-intervals of exponentially increasing size as follows. For $i\ge 1$, let
		\begin{align*}
		R_i^- &= [k^*-2^i\alpha, k^*-2^{i-1}\alpha) \\
		R_i^+ &=  (k^*+2^{i-1}\alpha,k^*+2^i\alpha] \\
		R_i &= R_i^-\cup R_i^+
		\end{align*}
		
		Then for any range-specific thresholds $t_i$ for $i\ge 1$, we will still bound the probability of the bad event as follows: 
		\begin{align}\label{eq.bad2}
		\Pr[\abs{\tilde k - k^*}>\alpha] &\le \sum_{i\ge 1}\Pr[\max_{k\in R_i} \set{\ell(k)-\ell(k^*)}>-t_i] + \sum_{i\ge 1}\Pr[\max_{k\in R_i}\set{Z_k-Z_{k^*}}\ge t_i]
		\end{align}
		
		We will re-define $U_j$ to denote the i.i.d random variables with mean zero by the alternative log-likelihood, and $S_m$ to denote the sum of $m$ i.i.d $U_j$  as follows:
		
		\begin{align*}
		U_j &= \begin{cases}
		- \log \frac{2P_0(x_j)}{(P_0+P_1)(x_j)}+D_{KL}(P_0||\frac{P_0+P_1}{2}), & j<k^* \\
		- \log \frac{2P_1(x_j)}{(P_0+P_1)(x_j)}+D_{KL}(P_1||\frac{P_0+P_1}{2}), & j\ge k^* 
		\end{cases}
		\\
		S_m &= \begin{cases} \sum_{k^*+m\le j<k^*} U_j, &m<0 \\ 
		\sum_{k^*\le j<k^*+m} U_j & m> 0 \end{cases}
		\end{align*}
		With these random variables, we can bound the empirical log-likelihood difference between the true change-point $k^\ast$ and any candidate $k$ by 
		\begin{align*}
		\frac{1}{2}[\ell(k)-\ell(k^*)] = \sum_{j=k}^{k^\ast}\log \frac{P_1(x_j)}{P_0(x_j)} &\leq \begin{cases}
		\sum_{j=k}^{k^*-1}U_j - (k^*-k)D_{KL}(P_0||\frac{P_0+P_1}{2}), & k< k^* \\
		\sum_{j=k^*}^{k-1}U_j - (k-k^*)D_{KL}(P_1||\frac{P_0+P_1}{2}), & k\ge k^* .
		\end{cases}
		\end{align*}
		
	    Then we bound each term in the first set of terms in~(\ref{eq.bad2}) for any $i\ge 1$ and threshold $t_i=2^{i-1}\alpha C_M$ as follows:
	    
	    \begin{align}
	    \Pr[\max_{k\in R_i} &\set{\ell(k)-\ell(k^*)}>-2^{i-1}\alpha C_M] \notag \\
	    &\le \Pr[\max_{k\in R_i^-}\set{\sum_{j=k}^{k^*-1}U_j-(k^*-k)D_{KL}(P_0||\frac{P_0+P_1}{2}))}>-2^{i-2}\alpha C_M] \notag \\
	    &\quad +\Pr[\max_{k\in R_i^+}\set{\sum_{j=k^*}^{k-1}U_j-(k-k^*)D_{KL}(P_1||\frac{P_0+P_1}{2})}>-2^{i-2}\alpha C_M] \notag \\
	    &\le \Pr[\max_{k\in[2^{i-1}\alpha]} \abs{S_{-k}} > 2^{i-2}\alpha C_M]  + \Pr[\max_{k\in[2^{i-1}\alpha]} \abs{S_{k}} > 2^{i-2}\alpha C_M] \\
	    &\le \frac{4\exp\left(-\frac{2^{i-4}\alpha C_M^2}{ C_M+ 32}\right)}{1-2\exp\left(-\frac{2^{i-4}\alpha C_M^2}{ C_M+ 32 }\right)} \label{eq.Thm7bound1} \\
	    &\le 8\exp\left(-\frac{2^{i-4}\alpha C_M^2}{ C_M+ 32}\right) \label{eq.Thm7bound2}
	    \end{align}
	    
	    where \eqref{eq.Thm7bound1} follows from an application of Corollary~\ref{cor.bern} with $\lambda_1=\lambda_2=2^{i-3}\alpha C_M$ and $v=4$. To apply Corollary~\ref{cor.bern}, we first need to check the conditions of Bernstein inequality. We shall show that for any $j$,
	    \begin{eqnarray}\label{norm2}
	    \E{\exp(|U_j|)-1-|U_j|} \le 2,
	    \end{eqnarray}
	    and then all conditions of Bernstein inequality are fulfilled.   
	    To prove this, let $Y_j$ be the i.i.d.~alternative log-likelihood ratio as follows:
	    \begin{align*}
	    Y_j &= \begin{cases}
	    - \log \frac{2P_0(x_j)}{(P_0+P_1)(x_j)}, & j<k^* \\
	    - \log \frac{2P_1(x_j)}{(P_0+P_1)(x_j)}, & j\ge k^* 
	    \end{cases}
	    \end{align*}
	    
	    Then it suffices to note that
	    \begin{align*}
	    	\E{\exp(|U_j|)}&=\E{\exp(|\log Y_j-\E{\log Y_j}|)}\\
	    	&\leq\E{\exp(\log Y_j-\E{\log Y_j})}
	    	+ \E{\exp(\E{\log Y_j}-\log Y_j)}\\
	    	&= \E{Y_j} e^{-C_{M}}+\frac{\E{1/Y_j}}{e^{-C_{M}}}\\
	    	&\leq e^{-C_M}+\frac{2}{e^{-C_M}},
	    \end{align*}
	    and the fact that $e^{-C_M}\in [1,2]$.
	    
	    It follows from direct calculations that the condition $\alpha\geq\frac{262}{C_M^2}\log \frac{128}{3\beta}>363/C_M^2$ implies $2\exp\left(-\frac{2^{i-4}\alpha C_M^2}{ C_M+ 32 }\right)<1/2$, which is used to simplify the denominator as in \eqref{eq.Thm7bound2}.
	     
	     We now consider the sum of these terms over all $i$, which will be needed for the final bound on Equation~(\ref{eq.bad2}).
	     
	     \begin{align}
	     \sum_{i\geq 1} \Pr[ \max_{k \in R_i} \{ \ell(k) - \ell(k^*) \} > -2^{i-2}\alpha C_M] & \leq \frac{16 \exp\left(-\frac{2^{-3}\alpha C_M^2}{ C_M+ 32}\right)}{1-\exp\left(-\frac{2^{-3}\alpha C_M^2}{ C_M+ 32}\right)} \notag \\
	     & \leq \frac{64}{3}\exp\left(-\frac{\alpha C_M^2}{262}\right). \label{eq.Thm7bound3}
	     \end{align}
	     
	     The first term in \eqref{eqnthm02eq3} in the theorem statement ensures that the expression above is bounded by $\beta/2$, as is required for the private algorithm. 
	     
	     For non-private MLE, we bound each term in the first set of terms in~(\ref{eq.bad2}) for any $i\ge 1$ and threshold $t_i=0$ as follows:
	     \begin{align}
	     \Pr[\max_{k\in R_i} &\set{\ell(k)-\ell(k^*)}>0] \notag \\
	     &\le \Pr[\max_{k\in R_i^-}\set{\sum_{j=k}^{k^*-1}U_j-(k^*-k)D_{KL}(P_0||\frac{P_0+P_1}{2})}>0] \notag \\
	     &\quad +\Pr[\max_{k\in R_i^+}\set{\sum_{j=k^*}^{k-1}U_j-(k-k^*)D_{KL}(P_1||\frac{P_0+P_1}{2})}>0] \notag \\
	     &\le \Pr[\max_{k\in[2^{i-1}\alpha]} \abs{S_{-k}} > 2^{i-1}\alpha C_M]  + \Pr[\max_{k\in[2^{i-1}\alpha]} \abs{S_{k}} > 2^{i-1}\alpha C_M] \\
	     &\le \frac{4\exp\left(-\frac{2^{i-3}\alpha C_M^2}{ C_M+ 8}\right)}{1-2\exp\left(-\frac{2^{i-3}\alpha C_M^2}{ C_M+ 8}\right)} \label{eq.Thm7bound4} \\
	     &\le 8\exp\left(-\frac{2^{i-3}\alpha C_M^2}{ C_M+ 8}\right) \label{eq.Thm7bound5}
	     \end{align}
	     
	      where \eqref{eq.Thm7bound4} follows from an application of Corollary~\ref{cor.bern} with $\lambda_1=\lambda_2=2^{i-2}\alpha C_M$ and $v=4$. It follows from direct calculations that the condition $\alpha\geq\frac{67}{C_M^2}\log \frac{128}{3\beta}>92/C_M^2$ implies $2\exp\left(-\frac{2^{i-3}\alpha C_M^2}{ C_M+ 8}\right)<1/2$, which is used to simplify the denominator as in \eqref{eq.Thm7bound5}. Then, we consider the sum of these terms over all $i$.
	      
	      \begin{align}
	      \sum_{i\geq 1} \Pr[ \max_{k \in R_i} \{ \ell(k) - \ell(k^*) \} > 0] & \leq \frac{16 \exp\left(-\frac{2^{-2}\alpha C_M^2}{ C_M+ 8}\right)}{1-\exp\left(-\frac{2^{-2}\alpha C_M^2}{ C_M+ 8}\right)} \notag \\
	      & \leq \frac{64}{3}\exp\left(-\frac{\alpha C_M^2}{67}\right). \label{eq.Thm7bound6}
	      \end{align}
	      
	      For $\alpha$ as in \eqref{eqnthm02eq4} in the theorem statement, the expression above is bounded by $\beta$, completing the accuracy proof for the non-private MLE.
	     
	     The calculations for the probability bounds for the Laplace noise terms are the same as those in Theorem \ref{thm02} with $C$ substituted by $2C_M$, which ends up with a probability no more than another $\beta/2$ under the condition $\alpha\ge \frac{2A\log (16/\beta)}{C_M\eps}$.

	     By Inequality~(\ref{eq.bad2}), this guarantees that $\Pr[\abs{\tilde k - k^*}>\alpha]\le \beta$ for the assumed ranges of $\alpha$ captured in Equation~\eqref{eqnthm02eq3} in the theorem statement, completing the proof.
\end{proof}


\section{Online private change-point detection}\label{s.online}

In this section, we give a new differentially private algorithm for change point detection in the online setting, \Online.  In this setting, the algorithm initially receives $n$ data points $x_1,\dots,x_n$ and then continues to receive data points one at a time. As before, the goal is to privately identify an approximation of the time $k^*$ when the data change from distribution $P_0$ to $P_1$. Additionally, we want to identify this change shortly after it occurs.

Our offline algorithm is not directly applicable because we do not know a priori how many points must arrive before a true change point occurs. To resolve this, \Online\ works like \AboveThresh, determining  after each new data entry arrives  whether it is likely that a change occurred in the most recent $n$ entries. When \Online\ detects a sufficiently large (noisy) partial log likelihood ratio $\ell(k)=\sum_{i=k}^j \log \frac{P_1(x_i)}{P_0(x_i)}$, it calls \Offline\ to privately determine the most likely change point $\tilde k$ in the window $\set{x_{j-n+1},\dots,x_j}$. 

Privacy of \Online\ is immediate from composition of \AboveThresh\ and \Offline, each with privacy loss $\eps/2$. As before, accuracy requires $X$ to be drawn from $P_0,P_1$ with some true change point $k^*$. This algorithm also requires a suitable choice of $T$ to guarantee that \Offline\ is called for a window of data that actually contains $k^*$. Specifically, $T$ should be large enough that the algorithm is unlikely to call \Offline\ when $j<k^*$ but small enough so that it is likely to call \Offline\ by time $j=k^*+n/2$. When both of these conditions hold, we inherit the accuracy of \Offline, with an extra $\log n$ factor arising from the fact that the data are no longer distributed exactly as in the change-point model after conditioning on calling \Offline\ in a correct window. 

With our final bounds, we note that $n\gg \frac{A}{C} \log(k^*/\beta)$ suffices for existence of a suitable threshold, and an analyst must have a reasonable approximation of $k^*$ in order to choose such a threshold. Otherwise, the accuracy bound itself has no dependence on the change-point $k^*$. 

{\centering
\begin{minipage}{\linewidth}
\begin{algorithm}[H]
\caption{Online private change-point detector : \Online($X, P_0, P_1,\epsilon, n, T$) 
}
\begin{algorithmic}
\State \textbf{Input:} database $X$, distributions $P_0,P_1$, privacy parameter $\eps$, starting size $n$, threshold $T$
\State Let $A=\max_x \log \frac{P_1(x)}{P_0(x)}-\min_{x'}\log \frac{P_1(x')}{P_0(x')}$
\State Let $\hat T = T+\Lap(4A/\eps)$ 
\For {each new data point $x_j, j\ge n$}
  \State Compute $\ell_j=\max_{j-n+1\le k\le j} \ell(k)$ 
  \State Sample $Z_j \sim \Lap(\frac{8A}{\epsilon})$ 
  \If {$\ell_j+Z_j>\hat T$}
    \State Output \Offline$(\set{x_{j-n+1},\dots,x_j},P_0,P_1,\eps/2,0,n)+(j-n)$ 
    \State Halt
  \Else
    \State Output $\bot$
  \EndIf
\EndFor
\end{algorithmic}\label{algo2}
\end{algorithm}
\end{minipage}
}

\begin{theorem}\label{thm.onlinepriv}
For arbitrary data $X$, \Online$(X,P_0,P_1,\eps,n,T)$ is $(\epsilon, 0)$-differentially private.
\end{theorem}

\begin{restatable}{theorem}{onlineacc}\label{thm.onlineacc}
For hypotheses $P_0,P_1$ such that $\Delta(\ell)<\infty$, a stream of data points $X$ with starting size $n$ drawn from $P_0,P_1$ with true change time $k^*\ge n/2$, privacy parameter $\eps>0$, and threshold $T\in[T_L,T_U]$ with
	\begin{eqnarray*}
	T_L&:=&2A\sqrt{2\log\frac{64k^*}{\beta}} - C+\frac{16A}{\eps}\log\frac{8k^*}{\beta},\\
	T_U&:=&\frac{nC}{2} - \frac{A}{2}\sqrt{n\log (8/\beta)}-\frac{16A}{\eps}\log\frac{8k^*}{\beta},
	\end{eqnarray*}	
we have that \Online$(X,P_0,P_1,\eps,n,T)$ is $(\alpha,\beta)$ accurate for any $\beta>0$ and 
	
	$$\alpha=\max\left\{\frac{16A^2}{C^2} \log\frac{32n}{\beta}, \frac{4A}{C\eps} \log\frac{8n}{\beta} \right\}.$$
	
	In the above expressions, $A=\Delta(\ell)$ and $C=\min\set{D_{KL}(P_0||P_1),D_{KL}(P_1||P_0)}$.
	
\end{restatable}

%

\begin{proof} 
We first give a range $[T_L,T_U]$ of thresholds that ensure that except with probability $\beta/4$, the randomly sampled data stream satisfies the following two conditions: 
\begin{enumerate}
\item For $T\ge T_L$, $\max_{k\in[j-n+1,j]}\ell(k)< T-\alpha'$ for every $j<k^*$.
\item For $T\le T_U$, $\max_{k\in[k^*-n/2,k^*+n/2)}\ell(k)>T+\alpha'$.
\end{enumerate}
When these conditions are satisfied, the \AboveThresh\ guarantee ensures that except with probability $\beta/4$, the randomness of the online algorithm ensures that it calls the offline algorithm on a window of data containing the true change-point. Then we will argue that our overall accuracy follows from the offline guarantee, where we will allow failure probability $\beta/2$.

We will get the first condition by taking a union bound over all windows tested before the change point of the probability that the maximum log-likelihood $\max_k \ell(k)$ for $n$ elements $X=(x_1,\dots,x_n)$ sampled from $P_0$ exceed a given threshold. To bound this probability, we first define the following random variables.
\begin{align*}
U_j &= -\log \frac{P_0(x_j)}{P_1(x_j)}+D_{KL}(P_0||P_1) \qquad\qquad
S_m = \sum_{1\le j\le m} U_j
\end{align*}
We note that each $\ell(k)$ is the sum of i.i.d. random variables, and that the maximum log-likelihood over $m$ consecutive elements is equal in distribution to $\max_{k\in[m]}S_k-kD_{KL}(P_0||P_1)$.  This yields the first inequality below. Inequality \eqref{eq.test} comes from applying Corollary \ref{cor.ott} with $\lambda_1=\lambda_2=2^{i-2}C+t/2$ and interval length $L=A$.
	\begin{align}
	\Pr\left[ \max_{1\le k \le n} \ell(k) > t\right] 
	\le & \sum_{i\ge 1}\Pr[\max_{k\in [2^{i-1},2^i)}\set{S_k-kD_{KL}(P_0||P_1)}>t ] \nonumber \\
	\le & \sum_{i\ge 1}\Pr[\max_{k\in [2^{i-1}]} S_k > 2^{i-1}C+t] \nonumber \\
	\le & \sum_{i\ge 1} \frac{2\exp(-(2^{i-2}C+t/2)^2/(2^{i-2}A^2))}{1-2\exp(-(2^{i-2}C+t/2)^2/(2^{i-2}A^2))}  \label{eq.test} \\
	\le & 4 \sum_{i\ge 1} \exp(-(2^{i-2}C+t/2)^2/(2^{i-2}A^2)) \label{eq.bla1} \\
	\le & 8 \exp(-(2^{-1}C+t/2)^2/(2^{-1}A^2)) \label{eq.bla2}
	\end{align}
	
Inequalities \eqref{eq.bla1} and \eqref{eq.bla2} follow by plugging in $t = 2A\sqrt{2\log\frac{64k^*}{\beta}} - C$.  This ensures that $1-2\exp(-(2^{i-2}C+t/2)^2/(2^{i-2}A^2)) \geq 1/2$, giving Inequality \eqref{eq.bla1}, and that the series is increasing exponentially in $i$, so we can collapse the sum with another factor of 2 by considering only $i=1$ as in Inequality \eqref{eq.bla2}.  This value of $t$ also ensures that the bound of Inequality \eqref{eq.bla2} is at most $\beta/(8k^*)$.
Taking the union bound over all the windows prior to the change-point, this shows that Condition 1 holds for $T_L=2A\sqrt{2\log\frac{64k^*}{\beta}} - C+\alpha'$ except with probability $\beta/8$.

To show that the second condition holds except with additional probability $\beta/8$, we consider the window of data with the first half of data drawn from $P_0$ and the second half drawn from $P_1$ and bound the probability that $\ell(k^*)$ in this window is less than a given threshold as follows. We note that $\ell(k^*)$ is the sum of $n/2$ i.i.d. random variables, so we define mean-zero random variables $V_j = -\log \frac{P_1(x_j)}{P_0(x_j)} +D_{KL}(P_1||P_0)$ and bound their sum using Hoeffding's inequality:

\begin{align}
\Pr[\max_{k^*-n/2\le k< k^*+n/2} \ell(k)<t] &\le \Pr[\ell(k^*)<t] \notag \\
&\le \Pr[\sum_{k\in[n/2]}V_j>nC/2-t] \notag \\
&\le \exp({-4(nC/2-t)^2}/(nA^2)) \label{eq.boundA}
\end{align}
Plugging in $t = \frac{nC}{2}-\frac{A}{2}\sqrt{n\log(8/\beta)}$ in this final expression ensures that \eqref{eq.boundA} $\leq \beta/8$.  This ensures that Condition 2 is satisfied except with probability $\beta/8$ for $T_U=nC/2-A\sqrt{2\log(8/\beta)}-\alpha'$. 


Then we can instantiate the \AboveThresh\ accuracy guarantee with privacy parameter $\eps/2$ and accuracy parameter $\beta/4$ to ensure that for $\alpha' = \frac{16A\log (8k^*/\beta)}{\eps}$ when Conditions 1 and 2 are satisfied, \AboveThresh\ will identify a window containing the true change-point except with probability $\beta/4$.  Combining this with the $\beta/4$ probability that Conditions 1 and 2 fail to hold when $T \in [T_L, T_U]$, we get that \Online\ calls \Offline\ in a window containing the change-point except with probability $\beta/2$ over the randomness of the data and of the online portion of the algorithm.

We next instantiate \Offline~with appropriate parameters to ensure that conditioned on being called in the correct window, it will output a $\tilde{k}$ that is within $\alpha$ of the true change-point $k^*$ with probability at most $\beta/2$.  We can then complete the proof by taking a union bound over all the failure probabilities.

Our offline accuracy guarantee requires data points sampled i.i.d. from $P_0$ before the change point and from $P_1$ thereafter, so it remains to show that conditioning on the event that we call the offline algorithm in a correct window does not harm the accuracy guarantee too much. 
For a window size $n$, change-point $k^*$, stream $X$ of at least $k^*+n/2$ data points, set of random coins required by \Online\ and its call to \Offline, and a stopping index $k>n/2$, let $N(k)$ denote the event that \Online\ calls \Offline\ on a window centered at $k$, and let $F(k)$ denote the event that \Offline\ on the window centered at $k$ fails to output an approximation within $\alpha$ of $k^*$. 

Our previous argument bounds the probability of all $N(k)$ for $k$ outside of a good range $G=(k^*-n/2,k^*]$, and our offline guarantee bounds the probability of $F(k)$ for any $k\in G$ as long as the data are drawn according to the change-point model. Then the overall probability of a bad event can be bounded as follows, where the probability is over the $X$ drawn from $P_0$ and $P_1$ with change-point $k^*$ and of the randomness of the algorithm:
\begin{align*}
\Pr[\abs{\tilde k-k^*}>\alpha] &= \sum_{k>n/2} \Pr[N(k)\cap F(k)] \\
&\le \sum_{k\not\in G}\Pr[N(k)] + \sum_{k\in G} \Pr[F(k)] 
\end{align*}
The first summation is at most $\beta/2$ by our previous arguments. By instantiation of Theorem~\ref{thm02} for \Offline\ with a $\beta/(2n)$ and $\eps/2$, the second summation is also bounded by $\beta/2$ when $\alpha= \max\{\frac{32A^2}{C^2} \log\frac{64n}{\beta}, \frac{8A}{C\eps} \log\frac{16n}{\beta} \}$.
\end{proof}



\section{Numerical studies}\label{s.sim}

We now report the results of Monte Carlo experiments designed to validate the theoretical results of previous sections.  We only consider our accuracy guarantees because the nature of differential privacy provides a strong worst-case guarantee for all hypothetical databases, and therefore is impractical and redundant to test empirically. Our simulations consider both offline and online settings for two canonical problems: detecting a change in the mean of Bernoulli and Gaussian distributions.


We begin with the offline setting to verify performance of our \Offline\ algorithm.  We use $n=200$ observations where the true change occurs at time $k^{*} = 100$. This process is repeated $10^4$ times. For both the Bernoulli and Gaussian models, we consider the following three different change scenarios, corresponding to the size of the change and parameter selection for \Offline. For each of these cases, we consider privacy parameter $\eps=0.1,0.5,1,\infty$, where $\eps=\infty$ corresponds to the non-private problem, which serves as our baseline. The results are summarized in \Cref{fig:offline}, which plots  the empirical probabilities $\beta = \Pr[|\tilde{k}-k^\ast|>\alpha]$ as a function of $\alpha$.

\vspace{-.5em}\begin{itemize}\setlength\itemsep{0em}
\item[(A)] \textbf{Large change.} Bernoulli model: detecting a change from $p_0=0.2$ to $p_1=0.8$. Gaussian model: detecting a change from $\mu_0=0$ to $\mu_1=1$.
\item[(B)] \textbf{Small change.} Bernoulli model: detecting a change from $p_0=0.2$ to $p_1=0.4$. Gaussian model: detecting a change from $\mu_0=0$ to $\mu_1=0.5$.
\item[(C)] \textbf{Misspecified change} Bernoulli model: algorithm tests for change from $p_0=0.2$ to $p_1=0.4$ when true distributions have $p_0=0.2$ and $p_1=0.8$. Gaussian model: algorithm tests for change from $\mu_0=0$ to $\mu_1=0.5$ when true distributions have $\mu_0=0$ and $\mu_1=1$.
\end{itemize}

\vspace{-.5em}

\Cref{fig:offline} highlights three positive results for our algorithm when data is drawn from Bernoulli or Gaussian distributions: accuracy is best when the true change in data is large (plots a and d) compared to small (plots b and e), accuracy deteriorates as $\eps$ decreases for stronger privacy, and the algorithm performs well even when the true change is larger than that hypothesized (plots c and f). This figure  emphasizes that our algorithm performs well even for quite strong privacy guarantees ($\eps<1$). The misspecified change experiments bolster our theoretical results substantially, indicating that our hypotheses can be quite far from the distributions of the true data and our algorithms will still identify a change-point accurately.

\begin{center}
\begin{minipage}{.9\linewidth}
\begin{figure}[H]
	\centering
	\subfloat[][Bernoulli, large change]{\includegraphics[width=.33\textwidth]{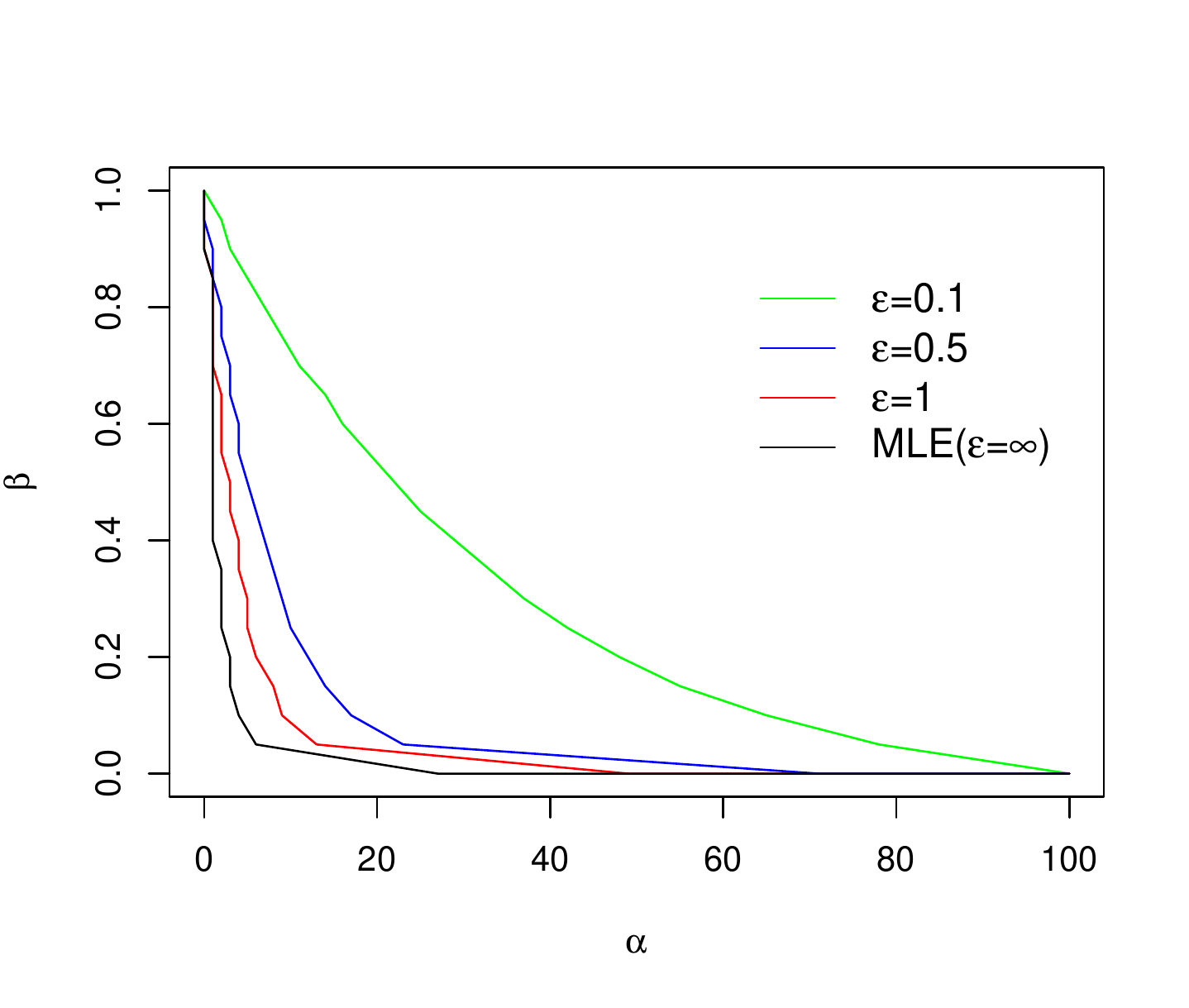}}
	\subfloat[][Bernoulli, small change]{\includegraphics[width=.33\textwidth]{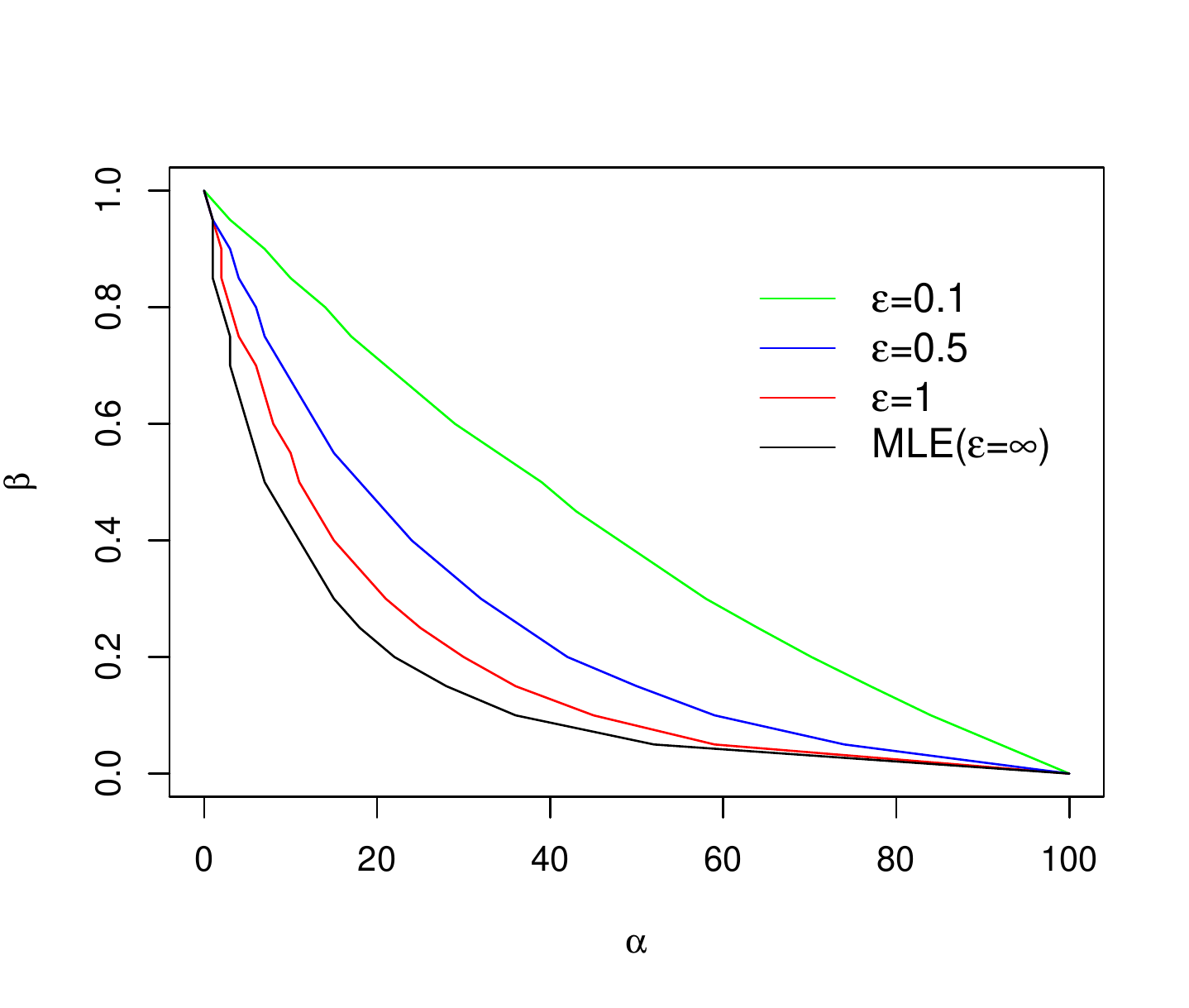}}
	\subfloat[][Bernoulli, misspecified change]{\includegraphics[width=.33\textwidth]{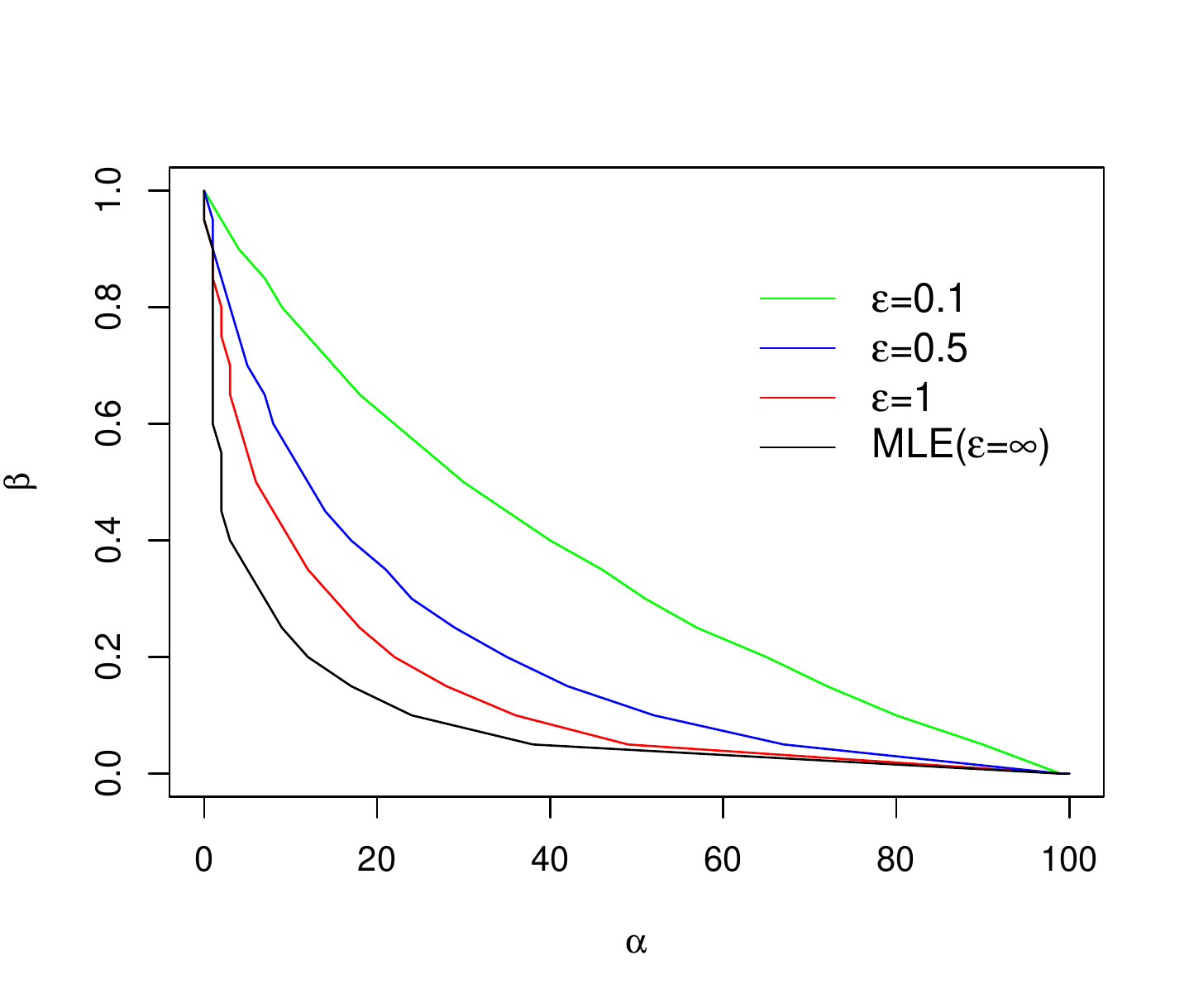}}\\
	\subfloat[][Gaussian, large change]{\includegraphics[width=.33\textwidth]{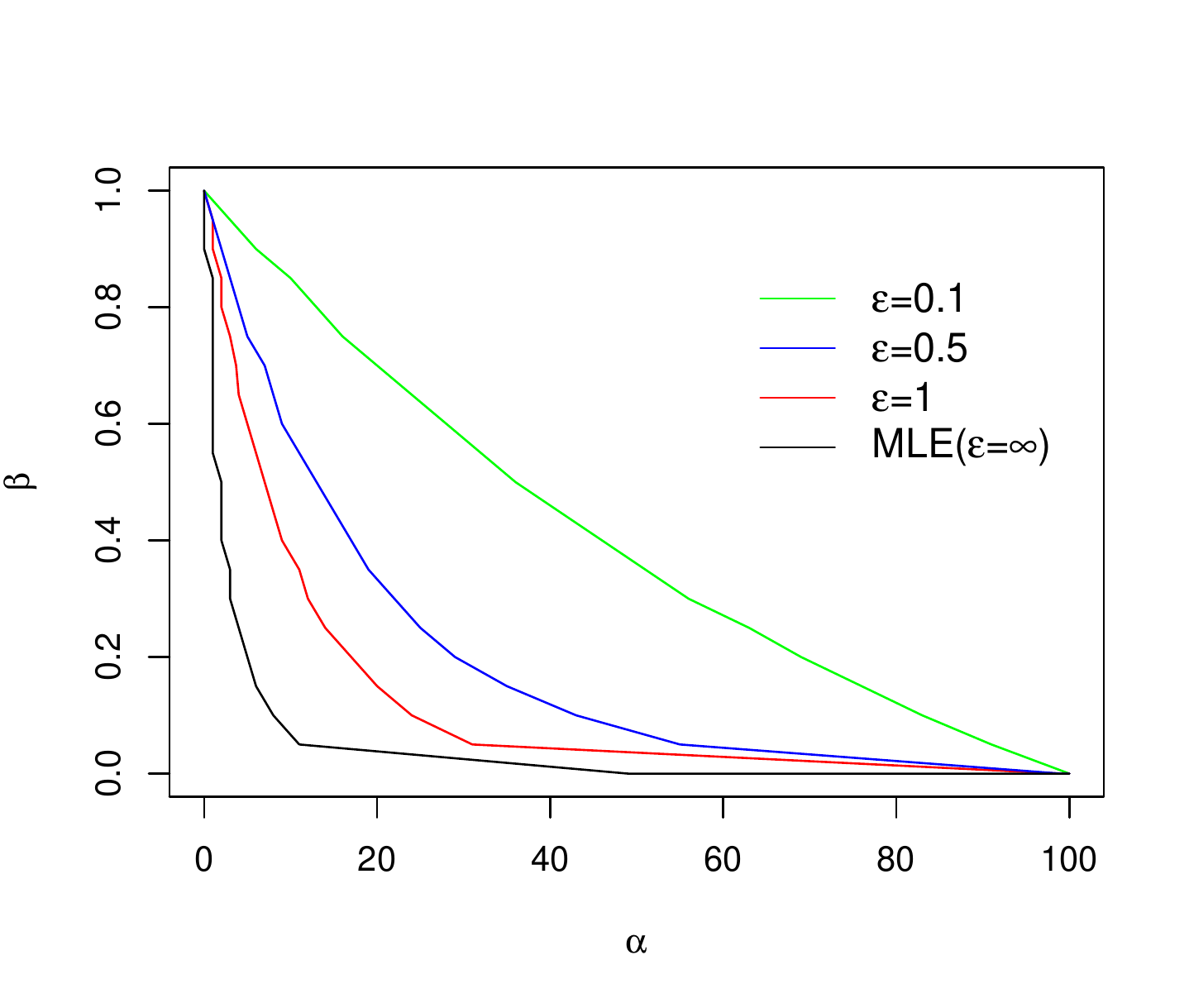}}
	\subfloat[][Gaussian, small change]{\includegraphics[width=.33\textwidth]{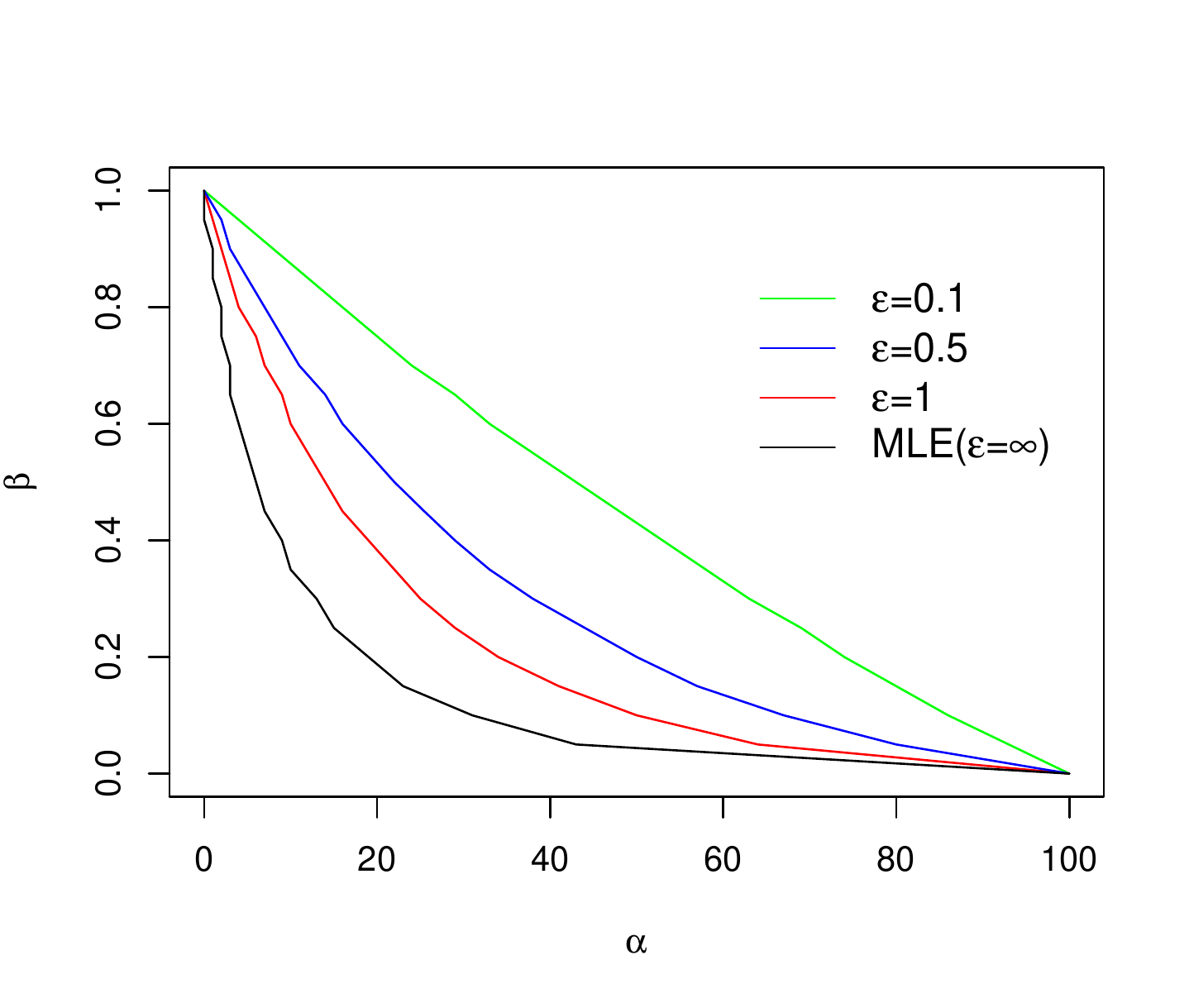}}
	\subfloat[][Gaussian, misspecified change]{\includegraphics[width=.33\textwidth]{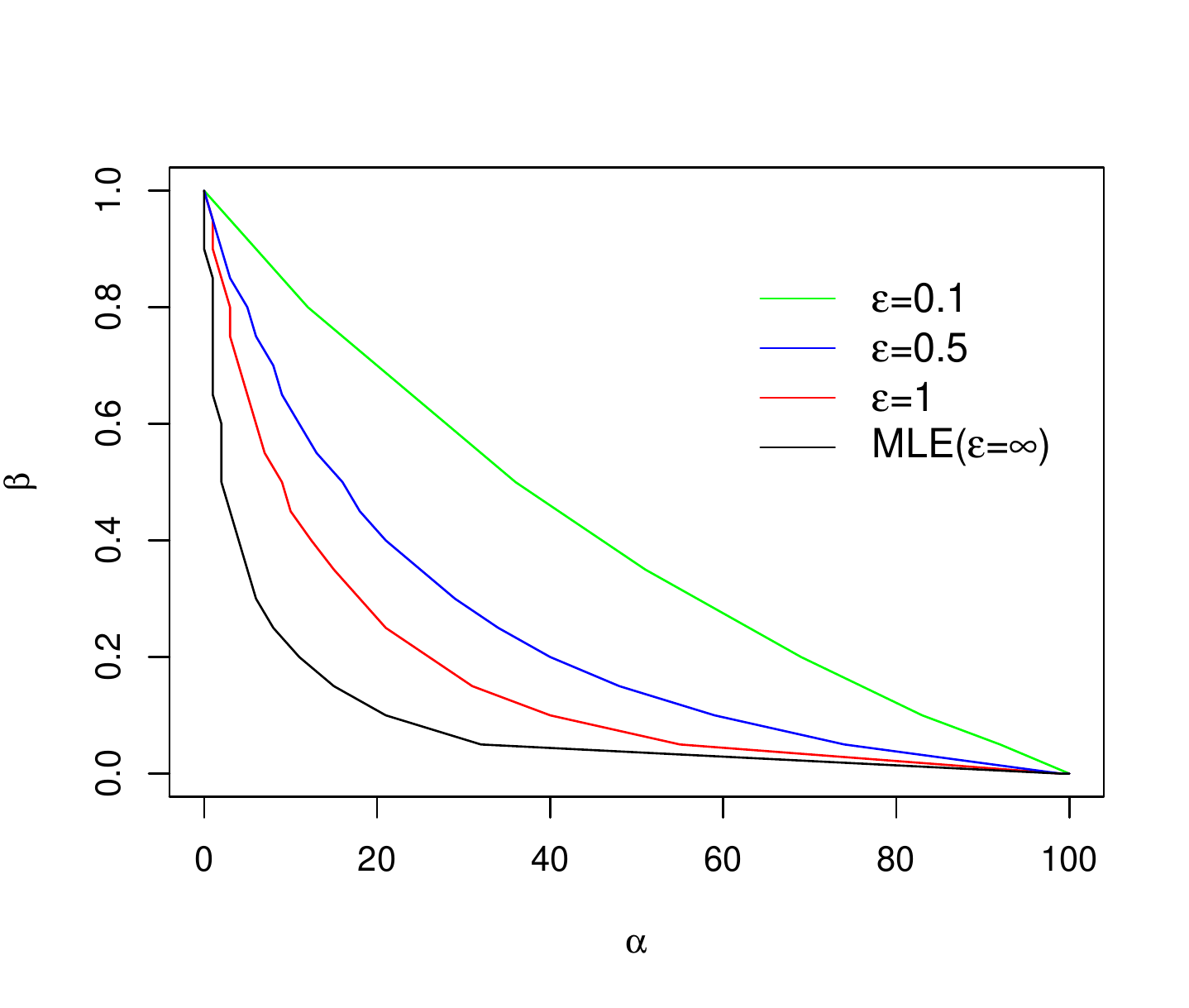}}
	\caption{\small Accuracy for large change, small change, and misspecified change Monte Carlo simulations with Bernoulli and Gaussian data. Each simulation involves $10^4$ runs of \Offline\ with varying $\eps$ on data generated by 200 i.i.d. samples from appropriate distributions with change point $k^*=100$. 
	} \label{fig:offline}
\end{figure}
\end{minipage}
\end{center}

We also run Monte Carlo simulations of our online change-point detection algorithm \Online, when the data points arrive sequentially and the true change occurs at time $k^{*} = 5000$. We choose the appropriate threshold $T$ by setting a constraint that an algorithm must have positive and negative false alarm rates both at most $0.1$. The range of threshold $T$ for the online algorithm needs to be non-empty, which impacts our choice of sliding window size $n$.  Unfortunately the window size of $n=200$ used in the offline simulations is not sufficient for our online examples.  A larger window size is needed to detect smaller changes or under higher levels of noise.  For this reason, we choose window size $n=700$ and restrict our online simulations to the large change scenario (A) and privacy parameters $\epsilon=0.5,1,\infty$.



For the online simulations, we use several key ideas in Section \ref{s.online} to speed up the numerical search of the threshold $T$. On the one hand, the threshold $T$ cannot be too small, otherwise a false alarm will be likely. To control the false alarm rate of $0.10$ with up to $k^{*} = 5000$ sliding windows, a conservative lower bound of the threshold $T$ is the $1-0.10/5000= 0.99998$ quantile of the noisy versions of $W_{n}$ with $n=700$ under the {\it pre-change} distribution. On the other hand, the threshold $T$ cannot be too large, otherwise it will fail to detect a true change in any sliding windows of size $n=700$. A useful upper bound of the threshold $T$ is the $10\%$ quantile of the noisy versions of CUSUM statistics $W_{n} = \max_{1 \le k \le n} \ell_{k}$ with $n=700$ when the change occurs at time $350$, since it will guarantee that the online algorithms raise an alarm with probability at least $0.9$ during the time interval $[4650,5350]$.

Next, we simulate  $10^{6}$ realizations of the CUSUM statistics $W_{n} = \max_{1 \le k \le n} \ell_{k}$ with $n=700$ in both the pre-change and post-change cases. In each case, we speed up the computation of $W_{i}$ by using the recursive form $W_{i} = \max{W_{i-1}, 0} + \log(P_{1}(X_{i})/P_{0}(X_{i}))$ for $i\ge 1$. The empirical quantiles of the noisy versions of $W_{n}$ with $n=700$ under the pre- and post- change cases will yield the lower and upper bounds of the threshold $T$. When the range of the threshold $T$ is non-empty, we choose one that is closest to the upper bound. For the Bernoulli model, we use $T=220$ for all values of $\epsilon=0.5, 1, \infty$.  In the Gaussian model, our window size $n=700$ is not sufficient to ensure non-empty range of $T$ under false alarm rate $0.2$ for $\epsilon=0.5,1$, so we relax the false alarm constraints for these $\epsilon$ values and choose $T=180,150,100$ for $\epsilon=0.5,1,\infty$, respectively.  Figure \ref{fig:online1} (c) indeed shows that the false alarm rates are high in the Gaussian model with $\epsilon=0.5,1$.


Figure \ref{fig:online1} summarizes our online simulations results for both Bernoulli and Gaussian models using a sliding window size $n=700$ to detect a large change (scenario A) that occurs at time $k^{*} = 5000$. Suppose our online algorithm raises an alarm at time $j$ with the estimated change-point $\tilde{k}_{j}$ for the sliding window of the observations, $\{x_{j-n+1}, \cdots, x_{j}\}.$ Two probabilities are plotted: one is  the marginal probability of inaccurate estimation and false alarm, $\beta_1 = \Pr(|\tilde{k}_{j}-k^\ast|>\alpha \mbox{ or } k^\ast \notin (j-n+1, j))$, and the other is the conditional probability of inaccurate estimation conditioned on raising an alarm correctly, $\beta_2 = \Pr(|\tilde{k}_{j}-k^\ast|>\alpha | j-n+1 \le k^\ast \le j).$  As $\alpha \to \infty,$ the probability $\beta_1$ becomes the false alarm rate plus the error rate related to the Laplace noise in hypothesis testing. For both Bernoulli and Gaussian models, the right-hand side plots in Figure \ref{fig:online1} (b and d) suggest that the online accuracy conditioned on correctly raising an alarm is very similar to the offline accuracy. Our plots show that the primary challenge in the online setting is determining when to raise an alarm in a sequence of sliding windows of observations. Once such window is identified correctly, the offline estimation algorithm can be used to accurately estimate the change-point.
	
	\begin{center}
\begin{minipage}{.9\linewidth}
	\begin{figure}[H]
		\centering
		\subfloat[][Bernoulli, inaccurate estimation and false alarm]{\includegraphics[width=.4\textwidth]{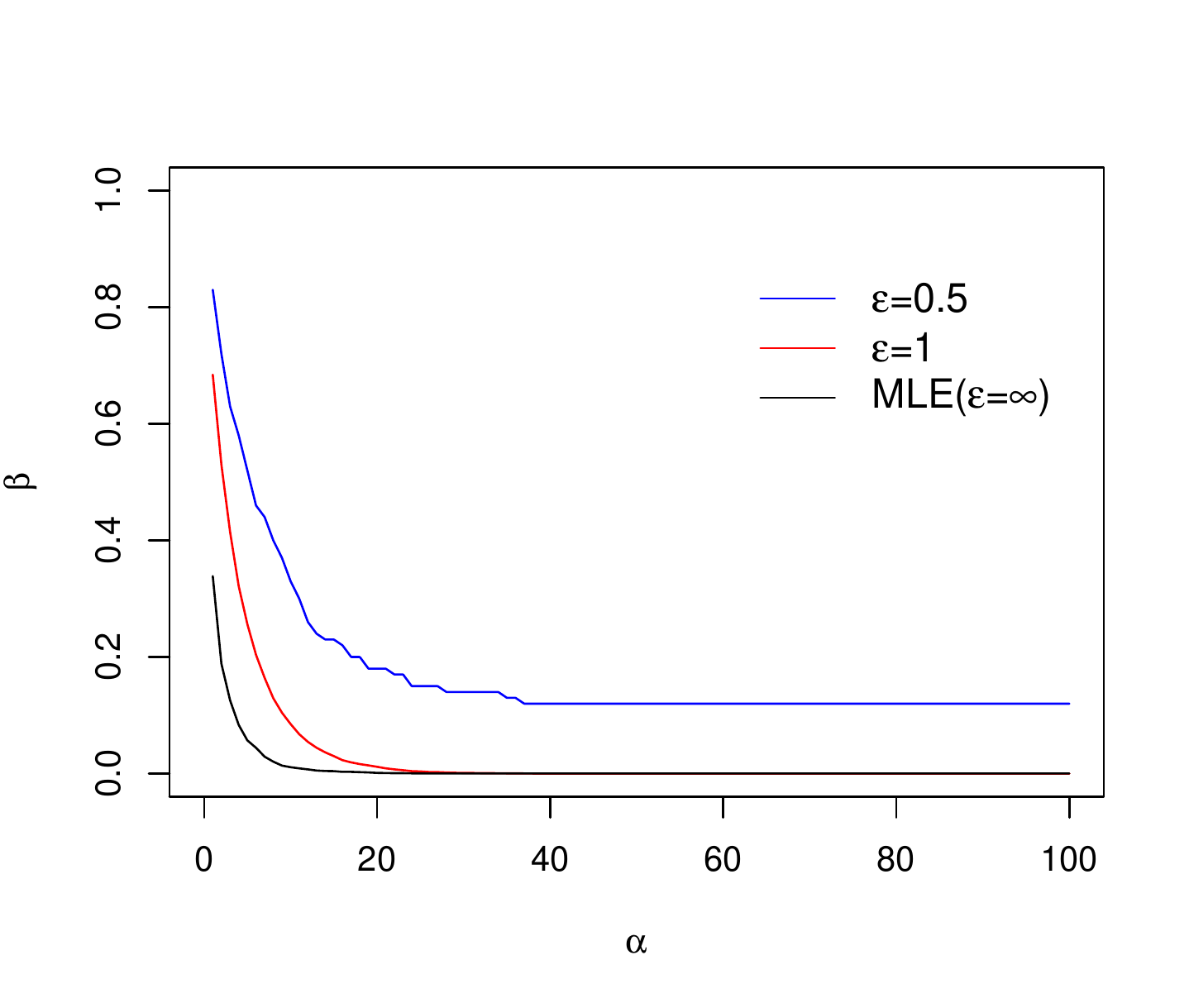}}\quad
		\subfloat[][Bernoulli, inaccurate estimation conditioned on no false alarm]{\includegraphics[width=.4\textwidth]{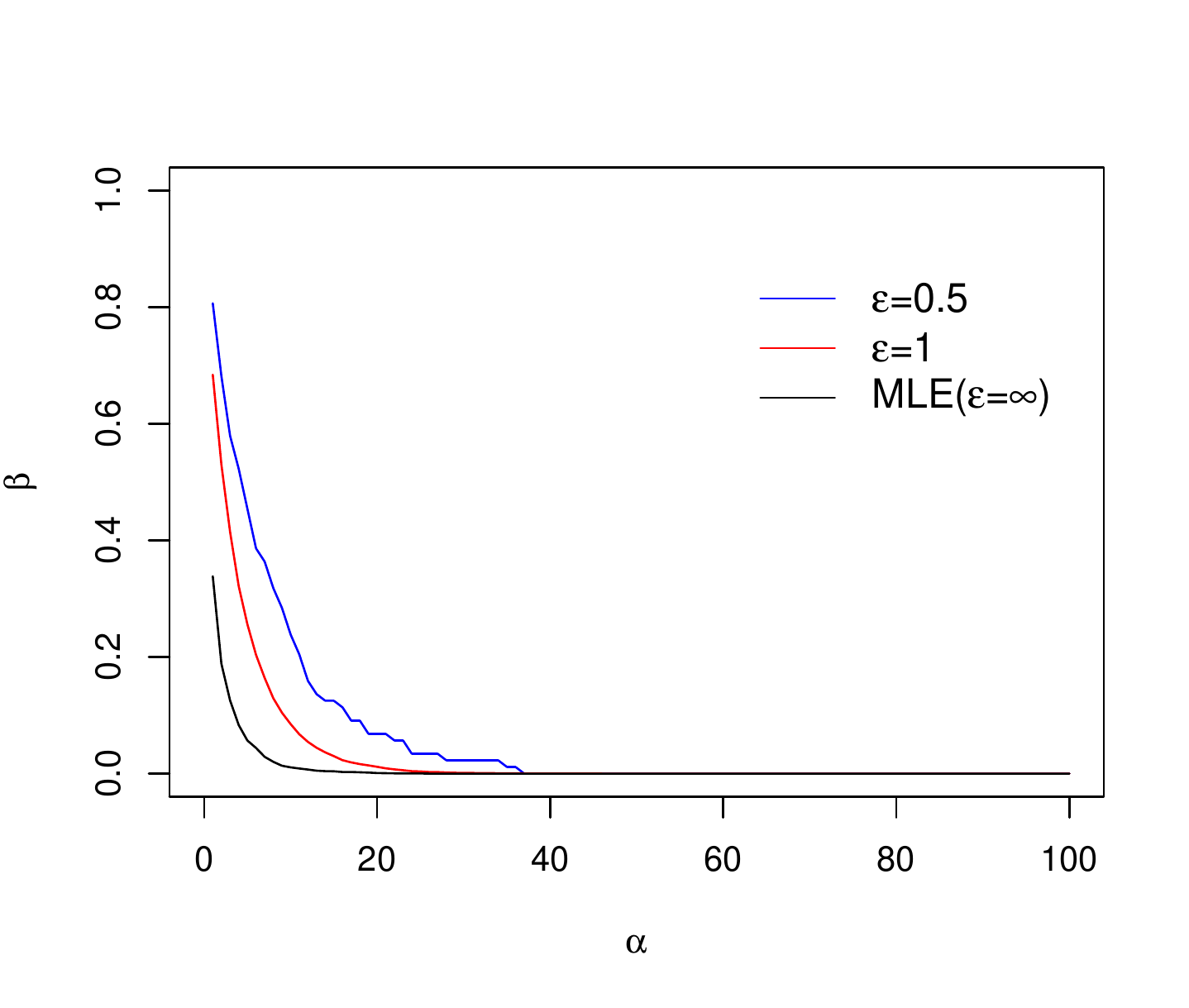}}\\
		\subfloat[][Gaussian, inaccurate estimation and false alarm]{\includegraphics[width=.4\textwidth]{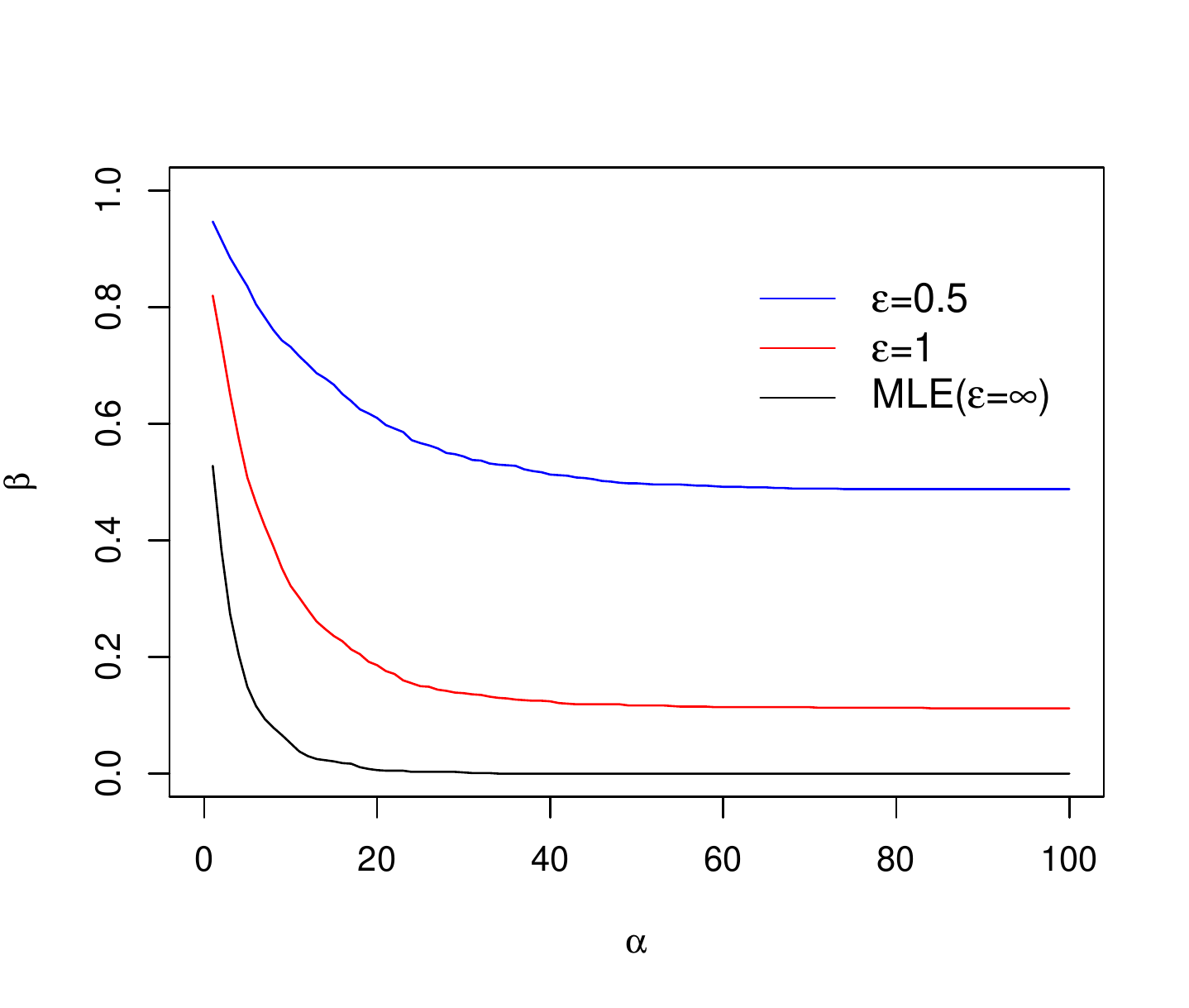}}\quad
		\subfloat[][Gaussian, inaccurate estimation conditioned on no false alarm]{\includegraphics[width=.4\textwidth]{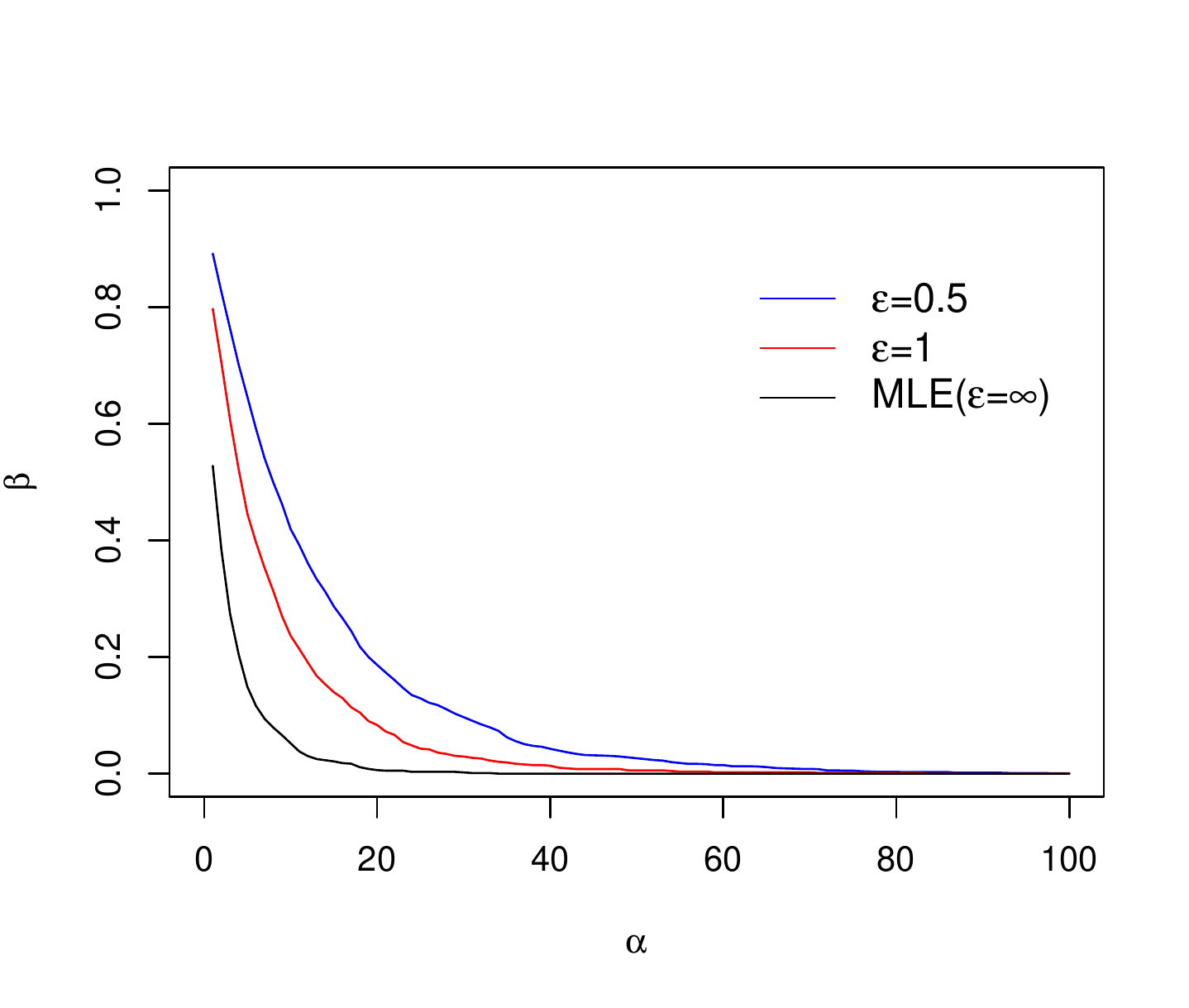}}
		\caption{\small Probability of inaccurate estimation and false alarm (left) and probability of accurate report conditioned on raising an alarm correctly (right) for Monte Carlo simulations with Bernoulli and Gaussian data. Each simulation involves $10^6$ runs of \Online\ with window size $n=700$ and varying $\eps$ on data generated by i.i.d. samples from appropriate distributions with change point $k^*=5000$. 	
} \label{fig:online1}
	\end{figure}
	\end{minipage}
	\end{center}

\bibliography{private}

\newcommand{\etalchar}[1]{$^{#1}$}
\begin{thebibliography}{VDVW96}

\bibitem[BP03]{Bai:Perron:2003}
J.~Bai and P.~Perron.
\newblock Computation and analysis of multiple structural change models.
\newblock {\em Journal of Applied Econometrics}, 18(1):1--22, 2003.

\bibitem[Car88]{Carlstein:1988}
E.~Carlstein.
\newblock Nonparametric change-point estimation.
\newblock {\em The Annals of Statistics}, 16(1):188--197, 1988.

\bibitem[Cha17]{chan:2017}
H.~P. Chan.
\newblock Optimal sequential detection in multi-stream data.
\newblock {\em The Annals of Statistics}, 45(6):2736--2763, 2017.

\bibitem[DMNS06]{DMNS06}
C.~Dwork, F.~McSherry, K.~Nissim, and A.~Smith.
\newblock Calibrating noise to sensitivity in private data analysis.
\newblock In {\em Proceedings of the 3rd Conference on Theory of Cryptography},
  TCC '06, pages 265--284, 2006.

\bibitem[DNR{\etalchar{+}}09]{DNRRV09}
C.~Dwork, M.~Naor, O.~Reingold, G.~N. Rothblum, and S.~P. Vadhan.
\newblock On the complexity of differentially private data release: efficient
  algorithms and hardness results.
\newblock In {\em Proceedings of the 41st ACM Symposium on Theory of
  Computing}, STOC '09, pages 381--390, 2009.

\bibitem[DR14]{dwork2014algorithmic}
Cynthia Dwork and Aaron Roth.
\newblock The algorithmic foundations of differential privacy.
\newblock {\em Foundations and Trends in Theoretical Computer Science},
  9(3--4):211--407, 2014.

\bibitem[Hin70]{Hinkley:1970}
D.~V. Hinkley.
\newblock Inference about the change-point in a sequence of random variables.
\newblock {\em Biometrika}, 57(1):1--17, 1970.

\bibitem[Kul01]{kulldorff:2001}
M.~Kulldorff.
\newblock Prospective time periodic geographical disease surveillance using a
  scan statistic.
\newblock {\em Journal of the Royal Statistical Society, Series A},
  164(1):61--72, 2001.

\bibitem[Lai95]{lai:1995}
T.~L. Lai.
\newblock Sequential changepoint detection in quality control and dynamical
  systems.
\newblock {\em Journal of the Royal Statistical Society, Series B},
  57(4):613--658, 1995.

\bibitem[Lai01]{lai:2001}
T.~L. Lai.
\newblock Sequential analysis: {s}ome classical problems and new challenges.
\newblock {\em Statistica Sinica}, 11(2):303--408, 2001.

\bibitem[Lor71]{lorden:1971}
G.~Lorden.
\newblock Procedures for reacting to a change in distribution.
\newblock {\em The Annals of Mathematical Statistics}, 42(6):1897--1908, 1971.

\bibitem[LR02]{Lund:2002}
R.~Lund and J.~Reeves.
\newblock Detection of undocumented changepoints: A revision of the two-phase
  regression model.
\newblock {\em Journal of Climate}, 15(17):2547--2554, 2002.

\bibitem[Mei06]{mei:2006a}
Y.~Mei.
\newblock Sequential change-point detection when unknown parameters are present
  in the pre-change distribution.
\newblock {\em The Annals of Statistics}, 34(1):92--122, 2006.

\bibitem[Mei08]{mei:2008a}
Y.~Mei.
\newblock Is average run length to false alarm always an informative criterion?
\newblock {\em Sequential Analysis}, 27(4):354--419, 2008.

\bibitem[Mei10]{mei:2010}
Y.~Mei.
\newblock Efficient scalable schemes for monitoring a large number of data
  streams.
\newblock {\em Biometrika}, 97(2):419--433, 2010.

\bibitem[Mou86]{moustakides:1986}
G.~V. Moustakides.
\newblock Optimal stopping times for detecting changes in distributions.
\newblock {\em The Annals of Statistics}, 14(4):1379--1387, 1986.

\bibitem[Pag54]{page:1954}
E.~S. Page.
\newblock Continuous inspection schemes.
\newblock {\em Biometrika}, 41(1/2):100--115, 1954.

\bibitem[Pol85]{pollak:1985}
M.~Pollak.
\newblock Optimal detection of a change in distribution.
\newblock {\em The Annals of Statistics}, 13(1):206--227, 1985.

\bibitem[Pol87]{pollak:1987}
M.~Pollak.
\newblock Average run lengths of an optimal method of detecting a change in
  distribution.
\newblock {\em The Annals of Statistics}, 15(2):749--779, 1987.

\bibitem[Rob66]{roberts:1966}
S.~W. Roberts.
\newblock A comparison of some control chart procedures.
\newblock {\em Technometrics}, 8(3):411--430, 1966.

\bibitem[She31]{shewhart:1931}
W.~A. Shewhart.
\newblock {\em Economic Control of Quality of Manufactured Product}.
\newblock D. Van Norstrand Company, Inc., 1931.

\bibitem[Shi63]{shiryaev:1963}
A.~N. Shiryaev.
\newblock On optimum methods in quickest detection problems.
\newblock {\em Theory of Probability \& Its Applications}, 8(1):22--46, 1963.

\bibitem[VDVW96]{van1996weak}
A.~W. Van Der~Vaart and J.~A. Wellner.
\newblock {\em Weak convergence}.
\newblock Springer, 1996.

\bibitem[ZS12]{Zhang:Siegmund:2012}
N.~Zhang and D.~O. Siegmund.
\newblock Model selection for high-dimensional, multi-sequence change-point
  problems.
\newblock {\em Statistica Sinica}, 22(4):1507--1538, 2012.

\end{thebibliography}
\bibliographystyle{alpha}

\end{document}